\documentclass{siamart251104}
\usepackage{fullpage}
\usepackage{dkmath}
\usepackage{dkmathssq}
\usepackage{amsmath,amssymb,amsfonts}
\usepackage{dsfont}
\usepackage{mathtools}
\usepackage{subcaption}
\usepackage{bm}
\usepackage{tikz}
\usepackage{tikzscale}
\usepackage[title]{appendix}
\usepackage{pgfplots}

\tikzset{perp_red/.pic={\draw[red,-latex,ultra thick,-](0,-.15)--(0,.15); 
    \node[text=black, font=\scriptsize] at(0,1){$#1$};},
    perp/.default={}}

\tikzset{perp_green/.pic={\draw[teal,-latex,ultra thick,-](0,-.15)--(0,.15); 
    \node[text=black, font=\scriptsize] at(0,1){$#1$};},
    perp/.default={}}

\tikzset{perp_blue/.pic={\draw[blue,-latex,ultra thick,-](0,-.2)--(0,.2); 
    \node[text=black, font=\scriptsize] at(0,1){$#1$};},
    perp2/.default={}}
    
\usetikzlibrary{shapes,shapes.multipart, shapes.geometric, arrows}
\usetikzlibrary{positioning}
\usetikzlibrary{decorations.markings}
\usetikzlibrary{decorations.pathreplacing}
\usetikzlibrary{plotmarks}
\usetikzlibrary{arrows.meta}
\usetikzlibrary{math,calc,patterns}

\pgfplotsset{compat=1.13}

\newtheorem{errest}{Error predictor}
\numberwithin{errest}{section}
\newsiamremark{remark}{Remark}

\def\kth{Department of Mathematics, KTH Royal Institute of Technology, Stockholm, Sweden}

\title{Adaptive singularity swap quadrature for near-singular layer potentials on axisymmetric surfaces}

\author{David Krantz%
  \thanks{\kth\,
  ({\tt davkra@kth.se}).}
    \and
Anna-Karin Tornberg%
  \thanks{\kth\,
 ({\tt akto@kth.se}).}
}

\date{\today}

\begin{document}

\maketitle

\begin{abstract}\label{s:abstract}When numerically evaluating layer potentials at target points close to the domain boundary, specialized quadrature techniques are required for accuracy because of rapid variations in the integrand. To efficiently achieve a prescribed error tolerance, we introduce an adaptive quadrature method for smooth axisymmetric surfaces in which all algorithmic choices are determined automatically from the requested error tolerance. Standard quadrature is used wherever it is sufficient, while a specialized near-quadrature correction is applied only for those target points where additional accuracy is required. This correction combines singularity swap quadrature in the azimuthal direction with adaptive refinement in the polar direction; on the resulting refined polar grid, either standard quadrature or singularity swap quadrature is used depending on the predicted quadrature error. The method is coupled to a standard quadrature based on the trapezoidal rule in the azimuthal direction and Gauss--Legendre quadrature in the polar direction, and is activated only when that rule is predicted to be insufficient. Quadrature and interpolation error predictors are derived using complex analysis and are used to control both activation and refinement. While each surface is assumed to be axisymmetric, the layer density and the overall geometry need not be, allowing applications to configurations with multiple smooth axisymmetric bodies and patchwise discretizations. Numerical examples for Laplace, Helmholtz, and Stokes layer potentials demonstrate reliable error control across a range of geometries, including multi-body configurations.

\end{abstract}


\section{Introduction}
\label{s:introduction}
We consider the numerical evaluation of layer potentials of the form
\begin{equation}
    u(\xx) = \int_S\frac{k(\xx,\yy)~\sigma(\yy)}{\left\|\yy-\xx\right\|^{2p}}\dS(\yy),\quad p=1/2,~3/2,~5/2,
    \label{eq:generic_layer_potential}
\end{equation}
where $\dS$ is the surface element measure of a smooth axisymmetric surface $S\subset\mathbb{R}^3$, $\sigma$ is a smooth scalar- or vector-valued function defined on $S$, and $k(\xx,\yy)$ is a smooth function originating from a fundamental solution of an elliptic partial differential equation (PDE). The evaluation (target) point $\xx\in\mathbb{R}^3$ may lie far from or close to the surface, but not on it.

Layer potentials of the form \eqref{eq:generic_layer_potential} arise in boundary integral formulations of elliptic PDEs and can be viewed as a convolution of the PDE's Green's function with the unknown ``density'' $\sigma$ over the boundary. Enforcing boundary conditions leads to an integral equation for $\sigma$, and once this equation is solved, the solution anywhere in the domain is obtained by evaluating the associated layer potential.

Such settings appear in many applications. At the microscale, for example, Stokes flow governs the motion of fluid–particle systems. These arise in the rheology of fiber and polymer suspensions \cite{roure2019}, suspensions of spheres \cite{reddig2013}, the design of new materials \cite{Hakansson2014,calabrese2020}, and the self-assembly of biological or synthetic particles \cite{wykes2016}. Axisymmetric geometries such as spheroids, rods, rings, and general bodies of revolution occur naturally in these contexts. Axisymmetric geometries are also important in acoustics and wave propagation governed by the Helmholtz equation, particularly in the design of absorbing surfaces and high-power loudspeakers. Boundary integral solvers have proven effective in both Stokes \cite{MALHOTRA2024112855,bagge2021,AFKLINTEBERG2016420,CORONA2017504,YAN2020109524} and Helmholtz \cite{YOUNG20124142,LIU2016226,HELSING2014686} settings. Nevertheless, accurate and error-controlled evaluation of layer potentials close to the surface remains a challenging and active area of research.

The difficulty is well known. As the target approaches the surface, the Green's function develops a sharp peak and the integrand varies rapidly. A straightforward remedy is to upsample the density, that is, interpolate $\sigma$ to a finer surface grid and apply a fixed \textit{standard quadrature} based on the surface discretization. However, the refinement needed to maintain accuracy grows dramatically as the target-surface distance shrinks, making this approach inefficient. This creates a need for \textit{special quadrature} methods designed specifically for \textit{nearly singular} integrals.

A large body of work addresses this need. In two dimensions, where boundary integrals reduce to line integrals, the problem is largely considered ``solved'', but in three dimensions it remains an active research area. Quadrature by expansion (QBX) \cite{KLOCKNER2013332,barnett2014,AFKLINTEBERG2016420} was first introduced in two dimensions and later extended to three dimensions. Since QBX integrates naturally with spherical-harmonic-based fast multipole methods, an integration of the two was envisioned. Yet achieving high efficiency proved difficult, and substantial effort was required before robust software for Laplace and Helmholtz problems became available \cite{WALA2019655}. Important advances include target-specific expansions \cite{Siegel2018}, which significantly reduce near-field cost and enable faster evaluations \cite{Wala2020}. 
Another promising direction, introduced in \cite{zhu2022} and improved in \cite{jiang2024}, transforms a surface integral into a collection of line integrals via Stokes' theorem. These nearly singular line integrals are then evaluated using the \textit{singularity swap quadrature} (SSQ) method, first proposed in \cite{AFKLINTEBERG2021} and subsequently extended in \cite{afKlinteberg2024,bao2024}, together with the recent stabilization in \cite{krantz2025stabilizingsingularityswapquadrature}. SSQ ``swaps'' the original nearly singular factor for a simpler one with matching singularities in the complex parameter plane. The remaining smooth factor is expanded in a suitable basis, and each basis function is then integrated analytically against the simplified near-singular term. These analytic integrals are typically computed via recurrence relations, whose stability may require additional care. Although the current Stokes'-theorem formulation is limited to Laplace layer potentials, the approach shows considerable potential.

Like QBX, the hedgehog or line-extrapolation method \cite{YING2006247,Bagge2023} 
exploits the fact that the layer-potential field remains smooth even when the integrand is sharply peaked. The potential is evaluated by standard quadrature at off-surface points farther away along a line from the surface, and is then extrapolated toward it, optionally including the surface value itself when available. Although appealing in its simplicity, the achievable accuracy is strongly dependent on the extrapolation distance, and the method's non-trivial parameter selection makes it challenging to use optimally.

Regularization methods replace the singular kernel by a smooth approximation, and recent work offers high-order approximations in the regularization parameter \cite{TlupovaBeale2025ArXiv}. One remaining difficulty lies in balancing the regularization error against the quadrature error, which increases as the regularization error decreases. For spherical geometries, vectorial spherical-harmonic-based quadrature has been explored \cite{CORONA2018327} and was recently extended to oblate and prolate spheroids \cite{crowder2025boundaryintegralequationanalysis}.

Axisymmetric geometries have also motivated significant work on specialized integral equation solvers and quadrature schemes. For Stokes flows, one strategy for making QBX more efficient is to precompute target-specific quadrature weights \cite{AFKLINTEBERG2016420,bagge2021}. This hides the QBX cost for on-surface evaluations. For off-surface, close evaluations, however, the expansion still must be recomputed because the target locations are not known in advance. This remains a computational bottleneck, and local panel-based forms of QBX could potentially mitigate this cost. For Helmholtz problems, the approach of \cite{YOUNG20124142} applies a discrete Fourier transform in the azimuthal direction, reducing the problem to Fourier integrals along the generating curve that are evaluated by analytic recursions. Helsing and Karlsson \cite{HELSING2014686} improved this by introducing analytic product integration in the axial direction based on \cite{HELSING2008}, and by addressing instabilities in the Fourier recurrences. These methods rely on the axisymmetric nature of the entire problem. While the quadrature ideas might extend to problems where only the geometry is axisymmetric, this was not explored.

Because special quadrature schemes are substantially more expensive than standard quadratures, it is desirable in any boundary integral method to activate them only when strictly necessary. This in turn requires a reliable mechanism for determining, for each target point, whether the standard quadrature is sufficiently accurate. In this work, that role is played by the \textit{error predictors} introduced in \cite{AFKLINTEBERG2022} and later refined for axisymmetric geometries in \cite{SORGENTONE2023}, both of which build on the asymptotic error analysis developed over a sequence of earlier studies (reviewed in \cite[Section 1]{AFKLINTEBERG2022}).

These predictors approximate the quadrature error by locating the relevant complex singularities of the integrand and inserting their positions into closed-form asymptotic formulas derived from the complex-variable framework of \cite{DONALDSON1972,ELLIOT2008,klinteberg2018}. Although they are not rigorous upper bounds, they are designed to be computationally inexpensive and practically effective at flagging when the standard quadrature is inadequate. 
In practice, they have proven remarkably accurate, making them well suited as a decision mechanism for when to invoke special quadrature.

\begin{figure}[t!]
\centering
\begin{tikzpicture}

\begin{axis}[%
width=0.6*4.521in,
height=3in,
at={(0.758in,0.405in)},
scale only axis,
point meta min=-14,
point meta max=0,
xmin=0,
xmax=1,
ymin=0,
ymax=1,
axis line style={draw=none},
ticks=none,
axis x line*=bottom,
axis y line*=left,
colormap={mymap}{[1pt] rgb(0pt)=(0.230033,0.298999,0.754002); rgb(1pt)=(0.438654,0.574065,0.953911); rgb(2pt)=(0.667541,0.779704,0.993653); rgb(3pt)=(0.865395,0.86541,0.865396); rgb(4pt)=(0.969053,0.721418,0.612449); rgb(5pt)=(0.908216,0.459395,0.358028); rgb(6pt)=(0.705998,0.0161265,0.150001)},
colorbar horizontal,
colorbar style={at={(0.5,1.03)}, anchor=south, xticklabel pos=upper, xlabel style={font=\color{white!15!black}}, xlabel={$\log_{10}(\textrm{Absolute error})$}},
colorbar sampled,
colorbar style={samples=8}
]
\end{axis}
\end{tikzpicture}%
\begin{subfigure}[t]{0.3\textwidth}
    \vspace{-22em}
    \centering
    %[trim={left bottom right top},clip]
    \includegraphics[trim={4.3cm 2cm 4.3cm 1.7cm},clip,width=\linewidth]{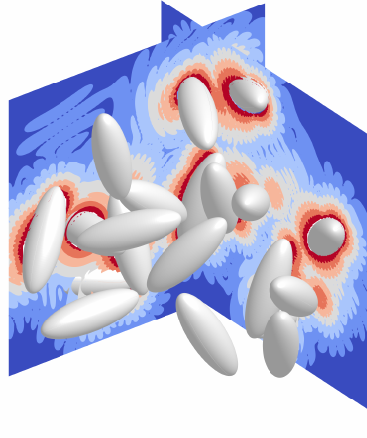}
    \caption{}
\end{subfigure}%
~ 
\begin{subfigure}[t]{0.3\textwidth}
\vspace{-22em}
    \centering
    \includegraphics[trim={4.3cm 2cm 4.3cm 1.7cm},clip,width=\linewidth]{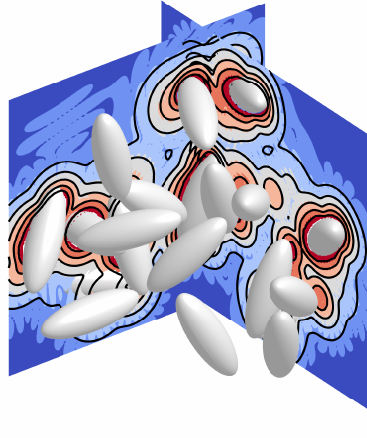}
    \caption{}
\end{subfigure}%
~ 
\begin{subfigure}[t]{0.3\textwidth}
\vspace{-22em}
    \centering
    \includegraphics[trim={4.3cm 2cm 4.3cm 1.7cm},clip,width=\linewidth]{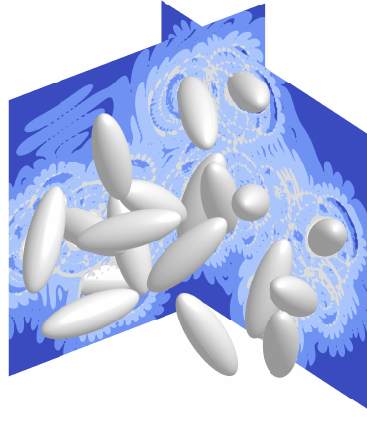}
    \caption{}
\end{subfigure}
\vspace{-3em}
\caption{Illustration of the S3Q workflow for evaluating a layer potential on a collection of smooth axisymmetric particles with prescribed tolerance $10^{-6}$. Panel (a) shows the loss of accuracy of standard quadrature near the particle surfaces. Panel (b) shows the quadrature error predictor, where the solid black contour lines indicate predicted error levels $10^{-10}$, $10^{-8}$, $10^{-6}$, $10^{-4}$, and $10^{-2}$, thereby identifying the regions requiring special quadrature. Panel (c) shows the error obtained using S3Q, demonstrating that the prescribed accuracy is achieved across all target points.}
\label{fig:s3q_flowchart}
\end{figure}

\subsection{Contributions and outline}\label{ss:contributions_outline}
The main contribution of this paper is an adaptive semi-analytic quadrature method for evaluating nearly singular layer potentials of the form \eqref{eq:generic_layer_potential} on smooth axisymmetric surfaces. The method combines singularity swap quadrature (SSQ) in the azimuthal direction, and when necessary, an analogous SSQ-based treatment in the polar direction, combined with adaptive refinement. Assuming only that the underlying surface discretization resolves the geometry and layer density, the method automatically selects all internal parameters. The \textit{only} required user input is the desired accuracy tolerance, although the order of the Gauss--Legendre rule used in the polar direction may be specified for efficiency. Because the resulting scheme uses SSQ in both surface-parameter directions, we refer to it as \textit{singularity swap surface quadrature} (S3Q). Figure \ref{fig:s3q_flowchart} illustrates the basic workflow.

The contributions of this paper, each of independent interest, are as follows:

\begin{itemize}
\item \textbf{SSQ for near-singular line integrals on closed curves in three dimensions.} We extend the SSQ method \cite{AFKLINTEBERG2021,afKlinteberg2024,bao2024} to near-singular line integrals over smooth closed curves embedded in three dimensions, discretized by the trapezoidal rule. This forms the one-dimensional building block used later in the azimuthal direction of the surface algorithm.
\item \textbf{Quadrature and interpolation error predictors.} We derive error predictors for adaptive refinement in the polar direction, including a new interpolation predictor obtained by combining the Hermite interpolation formula \cite[Theorem~11.1]{trefethen2019} with the complex-variable framework of \cite{AFKLINTEBERG2022}. This framework, previously used only for quadrature error prediction, is shown to yield accurate and practical interpolation error predictors for functions with near-singular branch-point behavior. These predictors are used to control the adaptive refinement in the polar direction.
\item \textbf{The S3Q algorithm for axisymmetric surface integrals}. By combining SSQ in both the azimuthal and polar directions, using the azimuthal SSQ as a building block and adaptive, predictor-based refinement in the polar direction, we obtain a fully automated and error-controlled scheme for evaluating single- and double-layer potentials for the 
Laplace, Helmholtz, and Stokes equations for target points close to smooth axisymmetric surfaces.
\end{itemize}

The outline of this paper is as follows. Section \ref{s:preliminaries} introduces notation and basic formulas. Section \ref{s:ssq} develops SSQ for near-singular line integrals on closed curves in three dimensions, providing the one-dimensional building block used in the azimuthal direction of S3Q. Section \ref{s:s3q} lifts this machinery to axisymmetric surface integrals and combines it with adaptive refinement in the polar direction to obtain the full S3Q algorithm. Section \ref{s:error_estimates} develops the practical quadrature and interpolation error predictors that drive this adaptive refinement. Section \ref{s:root_finding} discusses the computation of complex roots of distance functions used in the predictors. Section \ref{s:computational_complexity} analyzes complexity, and Section \ref{s:numerical_experiments} presents numerical results. We conclude in Section \ref{s:conclusions}. Appendices \ref{a:proofs}--\ref{a:rec_stability} contain auxiliary proofs, derivations of analytical root formulas, and details on stabilizing the recurrence relations used in the SSQ weights.

\section{Preliminaries}
\label{s:preliminaries}
\subsection{Geometry and parameterization}\label{ss:geometry_param}
Let $S$ be a smooth surface in $\mathbb{R}^3$ parameterized by $\ggamma:E\rightarrow\mathbb{R}^3$, where $E=\{E_1\times E_2\}\subset\mathbb{R}^2$. The layer potential \eqref{eq:generic_layer_potential} can then be written in parameter space as
\begin{equation}
    u(\xx) = \iint_E \frac{k(\xx,\ggamma(\theta,\varphi))~\sigma(\ggamma(\theta,\varphi))}{\left\|\ggamma(\theta,\varphi)-\xx\right\|^{2p}}~\left\|\frac{\partial\ggamma}{\partial \theta} \times \frac{\partial\ggamma}{\partial\varphi}\right\|\dphi\dtheta = \iint_E\frac{f(\theta,\varphi)}{\left\|\ggamma(\theta,\varphi)-\xx\right\|^{2p}}\dphi\dtheta,
    \label{eq:parameterized_layer_potential}
\end{equation}
where the function $f(\theta,\varphi)$ contains all the smooth components, and implicitly depends on the target point $\xx$. 

Throughout this paper we consider axisymmetric surfaces. We take $E=\{(\theta,\varphi)\in[0,\pi]\times[0,2\pi)\}$ and use the axisymmetric parametrization
\begin{equation}
    \ggamma(\theta,\varphi) = \left(a(\theta)\sin(\theta)\cos(\varphi),~a(\theta)\sin(\theta)\sin(\varphi),~b(\theta)\cos(\theta)\right),
    \label{eq:gamma_axi}
\end{equation}
where $a(\theta)>0$ and $b(\theta)>0$ are smooth analytic functions.

To simplify later notation, it is convenient to introduce a linear mapping between $t\in[-1,1]$ and a subinterval $[\tha,\thb]\subseteq[0,\pi]$
\begin{equation}
    \theta(t,\tha,\thb) = \frac{\tha+\thb}{2} + \thsc t,\quad \thsc = \frac{\thb-\tha}{2}.
    \label{eq:t2theta_map}
\end{equation}
with inverse
\begin{equation}
    t(\theta,\tha,\thb) = \frac{1}{\thsc}\left(\theta-\frac{\tha+\thb}{2}\right).
    \label{eq:theta2t_map}
\end{equation}
When $[\tha,\thb]=[0,\pi]$, we write simply $\theta(t)$ and $t(\theta)$.

Next, we define the squared-distance function for the surface $\ggamma(\theta,\varphi)=(\gamma_1(\theta,\varphi),\gamma_2(\theta,\varphi)),\gamma_3(\theta,\varphi))$ and a target point $\xx=(x,y,z)$,
\begin{equation}
    R^2(\theta,\varphi,\xx) \coloneqq \|\ggamma(\theta,\varphi)-\xx\|^2 = \left(\gamma_1(\theta,\varphi)-x\right)^2 + \left(\gamma_2(\theta,\varphi)-y\right)^2 + \left(\gamma_3(\theta,\varphi)-z\right)^2.
    \label{eq:R2}
\end{equation}
Using \eqref{eq:t2theta_map}, we write the layer potential compactly as
\begin{equation}
    \I[\Theta_p](\xx) = \I_t\I_\varphi[\Theta_p](\xx) = \iint_E \Theta_p(\theta(t,\tha,\thb),\varphi,\xx)\thsc\dphi\dt, \quad \Theta_p(\theta,\varphi,\xx) = \frac{f(\theta,\varphi)}{\left(R^2(\theta,\varphi,\xx)\right)^{p}}.
    \label{eq:Itheta}
\end{equation}
Here, the notation $\I_t$ and $\I_\varphi$ indicate integration in the $t$- and $\varphi$-direction, respectively. We will later use similar subscripts to distinguish in which direction an operator is applied.

\subsection{Discretization}\label{ss:discretization}
A natural approximation of \eqref{eq:Itheta} is the tensor-product rule formed from an $\nt$-point Gauss--Legendre quadrature rule on $t\in[-1,1]$ nodes $\{t_k\}_{k=1}^\nt$ and weights $\{\wt\}_{k=1}^\nt$, and an $\nphi$-point trapezoidal rule on $\varphi\in[0,2\pi)$ with nodes $\{\varphi_\ell\}_{\ell=1}^\nphi$ and weights $\{\wphi\}_{\ell=1}^\nphi$. The standard quadrature becomes
\begin{equation}
    \Q_{\nt,\nphi}[\Theta_p](\xx)=\Q_{t,\nt}\Q_{\varphi,\nphi}[\Theta_p](\xx) = \sum_{k=1}^\nt\sum_{\ell=1}^\nphi \frac{f(\theta(t_k),\varphi_\ell)\wphi\wt\thsc}{\left(R^2(\theta(t_k),\varphi_\ell,\xx)\right)^{p}}.
    \label{eq:reg_quad}
\end{equation}
with corresponding error
\begin{equation}
    \left|\E_{\nt,\nphi}[\Theta_p](\xx)\right| = \left|\I[\Theta_p](\xx) - \Q_{\nt,\nphi}[\Theta_p](\xx)\right|.
    \label{eq:reg_quad_err}
\end{equation}

\subsection{Complex roots of \boldmath{$R^2(\theta,\varphi,\xx)$}}\label{ss:roots_of_R2}
Although the integrand in \eqref{eq:Itheta} is smooth for all real $(\theta,\varphi)$ with target $x\notin S$, it develops sharp peaks when the target lies close to the surface. For a fixed value of one parametrization variable the function becomes singular, i.e., $R^2(\theta,\varphi,\xx)$ has roots at certain points in the complexified plane of the other variable.\footnote{$R^2(\theta,\varphi,\xx)$ is no longer a norm when evaluated with complex arguments, in which case we use the right-most expression in \eqref{eq:R2}.} These singularities bound the region of analyticity. Several complex-conjugate root pairs may exist, but their influence decays exponentially with distance from the real interval; only the closest pair is relevant for our purposes.\footnote{For some target locations, such as points on the symmetry axis of an axisymmetric surface, no such complex roots exist because the distance function remains strictly positive for all complexified angles.}

For fixed $\theta$, we denote by $\{\phiroot(\theta,\xx),\overline{\phiroot(\theta,\xx)}\}$ the pair closest to $E_2$, i.e.~the pair that governs the near-singular behavior. They satisfy,
\begin{equation}
    R^2\left(\theta,\phiroot(\theta,\xx),\xx\right) = R^2\left(\theta,\overline{\phiroot(\theta,\xx)},\xx\right) = 0.
\end{equation}

Similarly, fixing $\varphi$ and treating $\theta$ (or equivalently $t$) as complex, the relevant roots are \linebreak$\{\theta(t_0(\varphi,\xx)),\theta(\overline{t_0(\varphi,\xx)})\}$, where $\troot(\varphi,\xx)$ solves $R^2(\theta(t),\varphi,\xx)=0$ and lies closest to the interval $E_1$.

These roots can sometimes be expressed analytically, but in general they must be found by a one-dimensional numerical root search.
Their computation is discussed in Section~\ref{s:root_finding}. From this point forward, we write simply $\phiroot$ and $t_0$, leaving their dependencies implicit.

\section{Singularity swap quadrature for near-singular line integrals on closed curves in \boldmath{$\mathbb{R}^3$}}
\label{s:ssq}
In the axisymmetric setting, evaluation of the layer potential \eqref{eq:generic_layer_potential} naturally leads to a family of one-dimensional integrals obtained by fixing the polar parameter $\theta$ and integrating in the azimuthal direction. Each such slice is an integral over a circle in $\mathbb{R}^3$, and its accurate evaluation is essential for resolving nearly singular behavior when the target point lies close to the surface.

Motivated by this, we develop singularity swap quadrature (SSQ) for the more general task of evaluating integrals of the form \eqref{eq:generic_layer_potential} over an arbitrary smooth analytic closed curve $\Gamma$, following \cite{AFKLINTEBERG2021}. This extension is of independent interest, but in the present paper its main role is to provide the one-dimensional machinery used in the azimuthal direction of the surface quadrature method developed in Section \ref{s:s3q}.

Let $\boldsymbol{\xi}:[0,2\pi)\rightarrow\mathbb{R}^3$, where $\boldsymbol{\xi}(\phi)=(\xi_1(\phi),\xi_2(\phi),\xi_3(\phi))$, be an analytic, $2\pi$-periodic parametrization of $\Gamma$. To avoid confusion with the azimuthal parameter $\varphi$ used for axisymmetric surfaces, we reserve $\phi$ for the curve parameter.\footnote{When $\Gamma$ arises by fixing $\theta$ on an axisymmetric surface, one simply has $\boldsymbol{\xi}(\phi)=\ggamma(\theta,\phi)$.}

Define the squared-distance function
\begin{equation}
    R(\phi)^2 \coloneqq \|\boldsymbol{\xi}(\phi)-\xx\|^2 = \left(\xi_1(\phi)-x\right)^2+\left(\xi_2(\phi)-y\right)^2+\left(\xi_3(\phi)-z\right)^2
\end{equation}
Let $\phi_0$, $\overline{\phi_0}$ be the complex-conjugate roots of $R(\phi)^2$ closest to the interval $[0,2\pi)$, assuming they lie within the region of analyticity of $\boldsymbol{\xi}$.  In practice, $\phi_0$ can be found efficiently through a one-dimensional numerical root search, see \cite[Section~2.1]{afKlinteberg2024} for details.

The parametrized integral takes the form
\begin{equation}
    I_p = I_p(\xx) = \int_{0}^{2\pi} \frac{g(\phi)}{(R(\phi)^2)^p}\dpphi,\quad p=1/2,~3/2,~5/2,
    \label{eq:Ip}
\end{equation}
with $g(\phi)\coloneqq k(\xx,\boldsymbol{\xi}(\phi))\sigma_\Gamma(\phi)\|\boldsymbol{\xi}'(\phi)\|$, where $\sigma_\Gamma(\phi)$ denotes the restriction of the density $\sigma$ to the curve $\Gamma$.

The central idea of \cite{AFKLINTEBERG2021} is to transform the integral \eqref{eq:Ip} over a curve in $\mathbb{R}^3$ into an integral with a known singularity structure on a simple domain in $\mathbb{C}$. Once the singular points---namely the complex roots of the squared-distance function---are identified via analytic continuation, the integrand can be rewritten in terms of a simpler factor that exhibits the same local behavior at those points.
In the present setting the relevant singularities are $\phi_0$, $\overline{\phi_0}$, and the elementary function that vanishes at the same location is $|e^{i\phi}-e^{i\phi_0}|^2$.

``Swapping'' the singularity in \eqref{eq:Ip} using this factor yields
\begin{equation}
    I_p = \int_{0}^{2\pi} \frac{G(\phi)}{|e^{i\phi}-e^{i\phi_0}|^{2p}}\dpphi,
\end{equation}
where the ``SSQ numerator'' $G(\phi)$, defined by
\begin{equation}
    G(\phi)\coloneqq g(\phi)\frac{|e^{i\phi}-e^{i\phi_0}|^{2p}}{(R(\phi)^2)^p},
    \label{eq:ssq_G}
\end{equation}
is a smooth and $2\pi$-periodic function on $[0,2\pi)$ (in fact analytic in a typically large open neighborhood of $[0,2\pi)$).

Suppose $G$ is sampled at $n_\phi$ equispaced nodes, with $n_\phi$ even. Expanding $G$ in a truncated Fourier series and integrating termwise gives
\begin{equation}
    I_p \approx \sum_{k=-n_\phi/2}^{n_\phi/2-1} \widehat{G}_k(\phi_0)\underbrace{\int_{0}^{2\pi} \frac{e^{ik\phi}}{|e^{i\phi}-e^{i\phi_0}|^{2p}}\dpphi}_{\eqqcolon S_k^p(\phi_0)},
    \label{eq:Ip_approx_ssq}
\end{equation}
where $\widehat{G}_k$ are the Fourier coefficients of $G$ obtained via a fast Fourier transform (FFT) at $\mathcal{O}(n_\phi\log n_\phi)$ cost.

Thus $I_p$ is approximated by an interpolatory quadrature rule on the unit circle, with weights $S_k^p$ equal to the Fourier coefficients of $|e^{i\phi}-e^{i\phi_0}|^{-2p}$. A key advantage is that all near-singularity is confined to the Fourier basis integrals $S_k^p$, which are independent of the density $\sigma$. 
These integrals can be evaluated analytically, at cost $\mathcal{O}(n_\phi)$, using the recurrence relations stated in the following lemma. Its proof is given in Appendix \ref{a:proofs}.

In the lemma, and throughout the azimuthal SSQ construction, we take $\varphi_0$ to be the root with positive imaginary part. Choosing the conjugate root gives an equivalent SSQ formulation, differing only by normalization factors in intermediate expressions.

\begin{lemma}\label{lem:mu}Let $k\in\mathbb{Z}$, $p=\overline{p}+1/2$ with $\overline{p}\in\mathbb{Z}^+$, and $\phi_0=\alpha+i\beta$ with $\alpha,\beta\in\mathbb{R}$, $\beta>0$. Define $\chi=e^{-\beta}$, so that $0<\chi<1$. Then the Fourier basis integrals
\begin{equation}
    S_k^p(\phi_0) = \int_{0}^{2\pi}\frac{e^{ik\phi}}{|e^{i\phi}-e^{i\phi_0}|^{2p}}\dpphi
\end{equation}
can be computed as
\begin{equation}
    S_k^p(\phi_0) = \frac{2e^{ik\alpha}}{(1-\chi)^{2p-1}}\mu_k^p(\chi),\qquad k\in\mathbb{Z},
\end{equation}
where $\mu_{-k}^p(\chi)=\mu_k^p(\chi)$. For $m\geq0$, $\mu_m^p(\chi)$ is computed from the recurrences
\begin{equation}
    \mu_m^p(\chi) =
    \begin{cases}
        \dfrac{1+\chi^2}{\chi}\dfrac{2(m-1)}{2m-1}\mu_{m-1}^p(\chi) - \dfrac{2m-3}{2m-1}\mu_{m-2}^p(\chi), & p=1/2\text{ and } m=2,3,\dots,\\
        \dfrac{1+\chi^2}{2\chi}\mu_{m-1}^{p}(\chi) - \dfrac{(1-\chi)^2}{2\chi}\dfrac{p+m-2}{p-1}\mu_{m-1}^{p-1}(\chi), & p>1/2\text{ and } m=1,2,\dots.
    \end{cases}
    \label{eq:mu_rec}
\end{equation}
The initial values of $\mu_m^p(\chi)$ involve the complete elliptic integrals of the first and second kind,
\begin{equation}
    K(\chi^2) = \int_0^{\pi/2} \frac{\dtheta}{\sqrt{1-\chi^2\sin^2(\theta)}},\qquad E(\chi^2) = \int_0^{\pi/2}\sqrt{1-\chi^2\sin^2(\theta)}\dtheta,
    \label{eq:ellipticKE}
\end{equation}
and for $p=1/2,~3/2,~5/2$, are given by
\begin{align}
\mu_0^{1/2}(\chi) &= 2K(\chi^2), \quad \mu_1^{1/2}(\chi) = \frac{2}{\chi}\left(K(\chi^2)-E(\chi^2)\right),\label{eq:mu0p1}\\
\mu_0^{3/2}(\chi) &= \frac{2}{1+\chi}\left(\frac{2}{1+\chi}E(\chi^2) - (1-\chi)K(\chi^2)\right),\label{eq:mu0p3}\\
\mu_0^{5/2}(\chi) &= \frac{2}{3(1+\chi)^4}\left(8(1+\chi^2)E(\chi^2)-(1-\chi)(1+\chi)(5+3\chi^2)K(\chi^2)\right).\label{eq:mu0p5}
\end{align}
Moreover, $S_k^p(\phi_0)=\overline{S_{-k}^p(\phi_0)}$.
\end{lemma}

\begin{remark}\label{rem:laplace_coefficients}It turns out that the quantities $\mu_k^p(\chi)$ satisfy $\mu_k^p(\chi)=\frac{2}{\pi}(1-\chi)^{2p-1}b_p^{(k)}(\chi)$, where $b_p^{(k)}$ are the so-called Laplace coefficients appearing in multipole expansions of potentials. For instance, consider two points $\rr=(r,\theta,\varphi)$, $\rr'=(r',\theta',\varphi')$ in spherical coordinates, and let $h=r'/r$ and $\psi$ denote the angle between the two vectors. Then the potential
\begin{equation}
    \frac{1}{|\rr-\rr'|^{2p}} = \frac{1}{r^{2p}}~\frac{1}{\left(1-2h\cos(\psi)+h^2\right)^p}
\end{equation}
can be expanded as
\begin{equation}
    \frac{1}{\left(1-2h\cos(\psi)+h^2\right)^p} = \frac{1}{2}b_p^{(0)}(h) + \sum_{k=1}^\infty b_p^{(k)}(h)\cos(k\psi).
    \label{eq:laplace_expansion}
\end{equation}
Laplace introduced these functions in the context of celestial mechanics in 1785 \cite{laplace1785}. Although similar recurrences exist in that literature, they are difficult to access and not accompanied by proofs, so we retain the derivations used here.
\end{remark}

\begin{remark}[Accuracy and stability of recurrence formulas for $\mu_k^p$]\label{rem:mu_rec_stability}The recurrences for $\mu_k^p(\chi)$ in \eqref{eq:mu_rec} can be numerically unstable: the desired solution tends to zero as the number of forward steps $k\rightarrow\infty$, while the complementary homogeneous recurrence solution grows exponentially (at a rate that increases as $\chi\rightarrow0$). After a number of forward steps, the desired vanishing solution may thus be obscured by the exponentially increasing component, readily introduced by roundoff errors. To control this, we use error predictors to determine when the forward recurrence remains below the desired tolerance; otherwise, we switch the stable method of \cite{ARNOLDUS1984}, in which the recurrence is recast as a tridiagonal boundary value problem with a vanishing endpoint condition. Further details appear in Appendix \ref{a:rec_stability}.
\end{remark}

With this one-dimensional SSQ machinery in place, we now return to the surface problem and use it as the core building block of S3Q in both the azimuthal and polar directions.

\section{The S3Q method}
\label{s:s3q}
In this section, we extend the one-dimensional construction of Section \ref{s:ssq} from curves to smooth axisymmetric surfaces. The resulting method, which we call \textit{singularity swap surface quadrature (S3Q)}, gives a fully automated scheme for evaluating layer potentials of the form \eqref{eq:generic_layer_potential} to prescribed accuracy. We begin with an overview before giving the details.

Using notation from Section \ref{ss:geometry_param}, the parameterized layer potential is
\begin{equation}
    \mathcal{I}(\xx) = \int_0^\pi\int_0^{2\pi} \frac{f(\theta,\varphi)}{(R^2(\theta,\varphi,\xx))^p}\dphi\dtheta.
    \label{eq:s3q_Ithetaphi}
\end{equation}
For each fixed $\theta$, the inner integral in $\varphi$ is taken over the circle $\Gamma_\theta = \{\ggamma(\theta,\varphi) : \varphi\in[0,2\pi)\}$. This is precisely the type of one-dimensional integral treated by the SSQ machinery of Section \ref{s:ssq}. 

We therefore use SSQ to evaluate the azimuthal integral. The result of this step is a function of $\theta$ to be integrated. Because the computation of the inner integral relies on recurrence formulas, simple closed-form expressions for this function are not available. Although its dominant singular behavior is significantly weakened after integrating in $\varphi$, it can still be sharply peaked when the target lies close to the surface. For this reason, the outer polar integral requires an adaptive strategy equipped with reliable quadrature and interpolation error predictors, together with an SSQ treatment of the residual singularity in the cases where regular Gauss–Legendre quadrature is insufficient.

The S3Q method can be summarized as follows:
\begin{enumerate}
\setcounter{enumi}{-1}
\item \textbf{Target classification.} Given a prescribed tolerance, use the quadrature error predictor of \cite{SORGENTONE2023} to determine whether regular tensor-product quadrature is expected to be insufficient. If so, proceed with the remaining steps.
\item \textbf{Adaptive polar subdivision.} Use the quadrature and interpolation error predictors of Section \ref{s:error_estimates} to adaptively refine the $\theta$-interval.
\item \textbf{Azimuthal SSQ.} For each $\theta$-panel, apply SSQ to the closed curve $\Gamma_\theta$ whenever the standard trapezoidal rule does not meet the required accuracy. This yields an approximation of
\begin{equation}
    \mathcal{I}_\varphi(\theta) = \int_0^{2\pi} \frac{f(\theta,\varphi)}{(R^2(\theta,\varphi,\xx))^p}\dphi.
    \label{eq:s3q_Iphi}
\end{equation}
\item \textbf{Polar SSQ:} For each $\theta$-panel, apply SSQ to the generating curve whenever the standard Gauss--Legendre rule does not meet the required accuracy. The output is an approximation of
\begin{equation}
    \mathcal{I}(\xx) = \int_0^\pi\mathcal{I}_\varphi(\theta)\dtheta.
    \label{eq:s3q_Itheta}
\end{equation}
\end{enumerate}

This structure enables S3Q to resolve near-singular behavior efficiently, activating special quadrature only when needed and relying on inexpensive standard quadrature elsewhere. The following subsections describe each component in detail.

\subsection{Singularity swap quadrature in the azimuthal variable}\label{ss:analytic_reduction_to_line_integral}
The goal of this subsection is to apply the SSQ machinery from Section \ref{s:ssq} to evaluate the nearly singular azimuthal integral \eqref{eq:s3q_Iphi} for each fixed value of the polar value $\theta$. Our main analytical result is stated in Theorem \ref{thm:phi_int}, which follows from two preparatory lemmas, and the SSQ formulas of Section \ref{s:ssq}. 

Let $\phiroot=\phiroot(\theta)$ denote the complex-valued root of $\Rth$ with respect to $\varphi$. For axisymmetric surfaces, $\varphi_0$ is available in closed form via the formulas of \cite{SORGENTONE2023}, restated in Lemma \ref{lem:phi0}. A key result in the axisymmetric setting is that the factor $|e^{i\varphi}-e^{i\varphi_0}|^{2}/\Rth$ appearing in the ``SSQ numerator'', similar to \eqref{eq:ssq_G}, is independent of $\varphi$. 
This property allows the azimuthal integral \eqref{eq:s3q_Iphi} to collapse to an expression involving only the Fourier coefficients of $f(\theta,\cdot)$ and the Fourier basis integrals $\mu_k^p$.

\begin{lemma}[Root of $R^2$ in $\varphi$ \cite{SORGENTONE2023}]\label{lem:phi0}Let $\ggamma(\theta,\varphi)$ be parameterized as in \eqref{eq:gamma_axi}, and let $\Rth$ be defined by \eqref{eq:R2}. For a target point $\xx=(x,y,z)\in\mathbb{R}^3$ with $\rho^2=x^2+y^2>0$ and $\xx\notin\ggamma$, and any $\theta\in(0,\pi)$, the equation $R^2(\theta,\varphi,\xx)=0$ has the complex solutions
\begin{equation}
    \phiroot(\theta) = \alpha \pm i \beta(\theta),\qquad \alpha=\atantwo(y,x),\qquad \beta(\theta)=\ln\left(\frac{\lambda(\theta)+\sqrt{\lambda(\theta)^2-\rho^2\sin^2(\theta)}}{\rho\sin(\theta)}\right)>0,
    \label{eq:phi0}
\end{equation}
where
\begin{equation}
    \lambda(\theta) = \frac{\atilde(\theta)^2+\rho^2+\bigl(\btilde(\theta)-z\bigr)^2}{2a(\theta)}, \qquad\atilde(\theta)=a(\theta)\sin(\theta),\qquad\btilde(\theta)=b(\theta)\cos(\theta),
    \label{eq:lambda}
\end{equation}
with $\lambda(\theta)>\rho\sin(\theta)$.
\end{lemma}

In what follows, $\phiroot$ denotes the root in \eqref{eq:phi0} with positive imaginary part, consistent with the convention of Section \ref{s:ssq}.

\begin{lemma}\label{lem:indep_quotient}Assume the setting and notation of Lemma \ref{lem:phi0}. Then
\begin{equation}
    \frac{|e^{i\varphi}-e^{i\varphi_0}|^{2}}{\Rth} = \frac{1}{a(\theta)}\frac{1}{\lambda(\theta)+\sqrt{\lambda(\theta)^2-\rho^2\sin^2(\theta)}},
    \label{eq:indep_quotient}
\end{equation}
which is independent of $\varphi$.
\end{lemma}

\begin{proof}
Let $\alpha=\atantwo(y,x)$, so that $x=\rho\cos(\alpha)$ and $y=\rho\sin(\alpha)$. Then, by the definition of $\lambda(\theta)$ and $R^2(\theta,\varphi,\xx)$ we find
\begin{equation}
R^2(\theta,\varphi,\xx) = 2a(\theta)\bigl(\lambda(\theta)-\rho\sin(\theta)\cos(\varphi-\alpha)\bigr).
\end{equation}
Write $\varphi_0=\alpha+i\beta$ with $\beta>0$, and set $\chi=e^{-\beta}$. From Lemma \ref{lem:phi0},
\begin{equation}
\chi = \frac{\lambda(\theta)-\sqrt{\lambda(\theta)^2-\rho^2\sin^2(\theta)}}{\rho\sin(\theta)}.
\end{equation}
Since
\begin{equation}
|e^{i\varphi}-e^{i\varphi_0}|^2 = 1+\chi^2-2\chi\cos(\varphi-\alpha),
\end{equation}
substitution gives
\begin{equation}
|e^{i\varphi}-e^{i\varphi_0}|^2 = \frac{\lambda(\theta)-\sqrt{\lambda(\theta)^2-\rho^2\sin^2(\theta)}}{\rho^2\sin^2(\theta)}\,2\bigl(\lambda(\theta)-\rho\sin(\theta)\cos(\varphi-\alpha)\bigr).
\end{equation}
Dividing by $R^2(\theta,\varphi,\xx)$ and rationalizing gives the stated formula.
\end{proof}

With these preparations we now present the main analytical reduction of the layer potential to a one-dimensional integral in $\theta$. The theorem below shows that, after azimuthal integration and Fourier truncation, the remaining integrand always contains a square-root-type singularity, and for $p>1/2$ a weakened near-singularity compared to the original.

In the next subsection, we analyze this structure in detail to determine how \eqref{eq:layer_potential_integrated} is best integrated in the polar direction.

\begin{theorem}[Azimuthal SSQ reduction formula]\label{thm:phi_int}Let
\begin{equation}
    \ggamma(\theta,\varphi) = \left(a(\theta)\sin(\theta)\cos(\varphi),~a(\theta)\sin(\theta)\sin(\varphi),~b(\theta)\cos(\theta)\right),
\end{equation}
be a smooth axisymmetric surface with $a(\theta)$, $b(\theta)>0$. Consider the layer potential
\begin{equation}
    u(\xx) = \int_{0}^{\pi}\int_{0}^{2\pi}\frac{f(\theta,\varphi)}{\left\|\ggamma(\theta,\varphi)-\xx\right\|^{2p}}\dphi\dtheta,
    \label{eq:layer_potential_thm}
\end{equation}
where $2p\in\mathbb{Z}^+$ and target $\xx=(x,y,z)\notin\ggamma$ with $\rho^2=x^2+y^2$. Let $\phiroot(\theta)$ and $\lambda(\theta)$ be defined by \eqref{eq:phi0}--\eqref{eq:lambda}, with $\phiroot$ denoting the root with positive imaginary part, and define $\chi(\theta)=e^{-\Im(\phiroot(\theta))}$.

Then, using the Fourier expansion of $f(\theta,\cdot)$ truncated to $n_\varphi$ terms, the layer potential \eqref{eq:layer_potential_thm} is approximated by
\begin{equation}
    u(\xx) \approx \int_{0}^\pi \frac{1}{\sqrt{a(\theta)}}~\frac{1}{\left(\lambda(\theta)+\sqrt{\lambda(\theta)^2-\rho^2\sin^2(\theta)}\right)^{1/2}}~\frac{1}{\left(\Rlambda(\theta)\right)^{p-1/2}}~\Fcal(\theta)\dtheta,
    \label{eq:layer_potential_integrated}
\end{equation}
with
\begin{equation}
    \Fcal(\theta) = 2\sum_{k=-n_\varphi/2}^{n_\varphi/2-1} \fhat_k(\theta)e^{ik\Re(\phiroot(\theta))}\mu_k^p(\chi(\theta)),
    \label{eq:F}
\end{equation}
where $\fhat_k(\theta)$ denotes the discrete azimuthal Fourier coefficients of $f(\theta,\varphi)$ computed from the even $\nphi$ samples $f(\theta,\varphi_j)$.
The functions $\mu_k^p$ are defined by the recurrences in \eqref{eq:mu_rec}, and the reduced squared-distance function $\Rlambda$ is
\begin{equation}
    \Rlambda(\theta) = \left(a(\theta)\sin(\theta)-\rho\right)^2 + \left(b(\theta)\cos(\theta)-z\right)^2.
    \label{eq:Rlambda}
\end{equation}
\end{theorem}

\begin{proof}
Let $q(\theta)=\bigl(\lambda(\theta)^2-\rho^2\sin^2(\theta)\bigr)^{1/2}$. By Lemma \ref{lem:indep_quotient} and \eqref{eq:R2}
\begin{equation}
\frac{f(\theta,\varphi)}{\left\|\ggamma(\theta,\varphi)-\xx\right\|^{2p}} = \frac{1}{a(\theta)^p}\frac{1}{(\lambda(\theta)+q(\theta))^p}\frac{f(\theta,\varphi)}{|e^{i\varphi}-e^{i\phiroot(\theta)}|^{2p}}.
\label{eq:proof1}
\end{equation}
Approximating $f(\theta,\cdot)$ by its $n_\varphi$-term Fourier interpolant and applying Lemma \ref{lem:mu} gives
\begin{equation}
\int_0^{2\pi} \frac{f(\theta,\varphi)}{|e^{i\varphi}-e^{i\phiroot(\theta)}|^{2p}}\dphi\approx\frac{\Fcal(\theta)}{(1-\chi(\theta))^{2p-1}},
\label{eq:proof2}
\end{equation}
with $\mathcal{F}$ defined in \eqref{eq:F}. Since $\chi(\theta)=(\lambda(\theta)-q(\theta))/(\rho\sin(\theta))$, direct computation gives
\begin{equation}
\bigl(1-\chi(\theta)\bigr)^2 = \frac{R_\lambda^2(\theta)}{a(\theta)\bigl(\lambda(\theta)+q(\theta)\bigr)}
\end{equation}
where $R_\lambda^2$ is defined in \eqref{eq:Rlambda}. Hence
\begin{equation}
\frac{1}{\bigl(1-\chi(\theta)\bigr)^{2p-1}} = \left(\frac{a(\theta)\bigl(\lambda(\theta)+q(\theta)\bigr)}{\Rlambda(\theta)}\right)^{p-1/2}.
\label{eq:proof3}
\end{equation}
Substituting \eqref{eq:proof1}, \eqref{eq:proof2}, and \eqref{eq:proof3} into \eqref{eq:layer_potential_thm} gives \eqref{eq:layer_potential_integrated}.
\end{proof}

\subsection{Examination of integral in the polar direction}\label{ss:regularity}
Following the azimuthal SSQ reduction in Section \ref{ss:analytic_reduction_to_line_integral}, the layer potential is approximated by a one-dimensional integral in the polar parameter $\theta$. The purpose of this section is to analyze the structure of the reduced integrand in \eqref{eq:layer_potential_integrated} as a function of $\theta$, and to identify the features that determine the accuracy of its numerical evaluation. 

We write \eqref{eq:layer_potential_integrated} in the form
\begin{equation}
    u(\xx) = \int_0^\pi \frac{1}{\sqrt{a(\theta)}} \Lsq(\theta,\xx)\Fcal(\theta,p,\xx) 
    \left(\Rlambda(\theta,\xx)\right)^{-(p-1/2)} \dtheta,
    \label{eq:layer_potential_abc}
\end{equation}
where
\begin{equation}
    \Lsq(\theta,\xx) = \left(\lamb+\sqrt{\lamb^2-\rho^2\sin^2(\theta)}\right)^{-1/2},
    \label{eq:Lsq}
\end{equation}
with $\lambda$ defined in \eqref{eq:lambda}. The function $\Fcal$ is written with three arguments to emphasize its dependencies.

Among the factors in the integrand, the inverse $\Rlambda$ term is the most difficult to resolve numerically, as it exhibits a nearly singular behavior for close target points. Its singularity order is, however, reduced compared to that of the original square-distance function $R^2$, and for $p=1/2$ this factor is identically equal to one. The function $\Fcal(\theta,p,\xx)$, defined via recurrence relations, may have a logarithmic near singularity\footnote{The function $\Fcal$ is defined in \eqref{eq:F}, with $\mu_k^p$ computed via the recurrence relations \eqref{eq:mu_rec}. The initial values \eqref{eq:mu0p1}-\eqref{eq:mu0p5} involve the complete elliptic integral of the first kind $K(r^2)$. As the evaluation point $\xx$ approaches the surface, the quantity $\chi(\theta)=e^{-|\Im(\phiroot(\theta,\xx))|} \rightarrow 1^-$, and $K(\chi^2)$ diverges logarithmically.}. The factor $\Lsq(\theta,\xx)$ has a nearly singular derivative and, while remaining bounded in magnitude, is nevertheless challenging to approximate accurately by global polynomial expansions. The remaining factor $1/\sqrt{a(\theta)}$ is benign.

Figure \ref{fig:cheb_value_coeffs} illustrates the typical behavior of the functions appearing in \eqref{eq:layer_potential_abc} for close evaluation of the layer potential \eqref{eq:parameterized_layer_potential} on a spheroidal surface. Each factor is represented by its global Chebyshev expansion, and the decay of the corresponding coefficient magnitudes is shown. The results indicate that $\Lsq(\theta,\xx)$ and $\Fcal(\theta,p,\xx)$ require comparable polynomial degrees to be resolved to a given accuracy, whereas for $p>1/2$ the factor $(\Rlambda(\theta,\xx))^{-(p-1/2)}$ exhibits significantly slower coefficient decay and therefore requires much higher polynomial degree.

For $p>1/2$, the function $(\Rlambda(\theta,\xx))^{-(p-1/2)}$ has a singularity at $\thetarootlambda\in\mathbb{C}$, a complex root of $\Rlambda$. This singularity can be regularized by multiplication with $|\theta-\thetarootlambda|^{2p-1}$. As shown in Figure \ref{fig:cheb_coeffs_p3}, this regularization dramatically reduces the number of Chebyshev modes required to resolve the factor, confirming that the slow decay observed in the unregularized factor is directly associated with the complex root structure of $\Rlambda$.

The functions $\Lsq(\theta,\xx)$ and $\Fcal(\theta,p,\xx)$ possess branch-point and logarithmic singularities, respectively, which limit the decay of global polynomial expansions, as observed in Figures~\ref{fig:cheb_coeffs_p1} and \ref{fig:cheb_coeffs_p3}. These singularities are integrable and localized near the same complex locations that govern the behavior of $\Rlambda$. While they affect the global approximation rate of the integrand, they do not introduce stronger singular behavior than that already present in the $\Rlambda$ factor.

For this reason, the numerical treatment focuses on $\Rlambda$. We employ an adaptive subdivision of the polar integration interval, and for $p>1/2$, an SSQ-based approach is used to handle the $\Rlambda$ factor, while for $p=1/2$ standard Gauss--Legendre quadrature is applied on each subinterval. Details of the adaptive discretization are given in Section \ref{ss:local_refinement}. We now proceed to describe the special quadrature construction for $p>1/2$.

\begin{figure}[!t]
\centering
\definecolor{mycolor1}{rgb}{1.00000,0.00000,1.00000}%
\begin{tikzpicture}

\begin{axis}[%
width=4.521in,
height=3.566in,
at={(0.758in,0.481in)},
scale only axis,
xmin=1,
xmax=2,
ymin=1,
ymax=2,
axis line style={draw=none},
ticks=none,
legend style={legend cell align=left, align=left, draw=white!15!black},
legend columns=5
]
\addplot [color=blue, line width=2.0pt]
  table[row sep=crcr]{%
0	0\\
};
\addlegendentry{$\Lsq$}

\addplot [color=mycolor1, line width=2.0pt]
  table[row sep=crcr]{%
0	0\\
};
\addlegendentry{$\Fcal$}

\addplot [color=red, line width=2.0pt]
  table[row sep=crcr]{%
0	0\\
};
\addlegendentry{$(\Rlambda)^{-(p-1/2)}$}

\addplot [color=black!30!green, line width=2.0pt]
  table[row sep=crcr]{%
0	0\\
};
\addlegendentry{$|\theta-\theta_0^\lambda|^{2p-1}(\Rlambda)^{-(p-1/2)}$}

\addplot [color=black, line width=2.0pt, dashed]
  table[row sep=crcr]{%
0	0\\
};
\addlegendentry{$\theta=\Re(\thetarootlambda)$}

\end{axis}
\end{tikzpicture}%
\begin{subfigure}[t]{0.45\textwidth}
\vspace*{-8.3cm}\includegraphics[width=\textwidth]{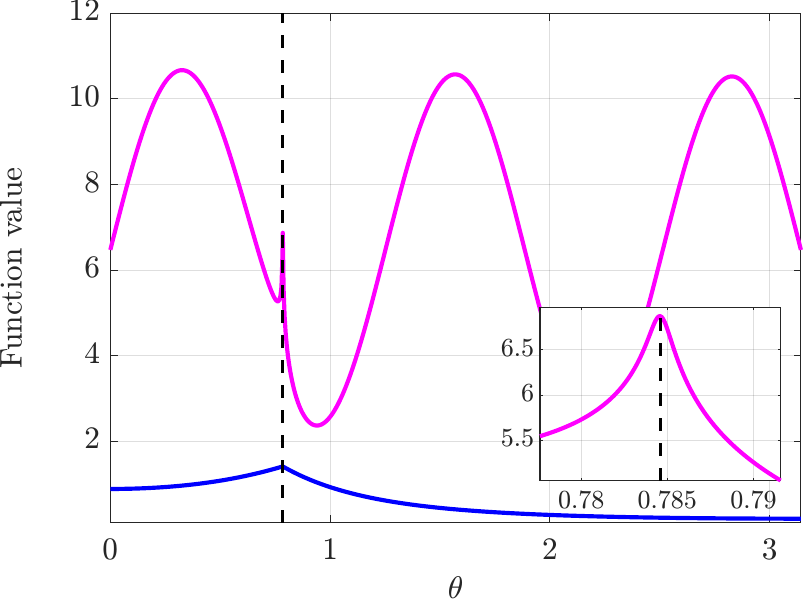}
\caption{$p=1/2$}
\label{fig:cheb_value_p1}
\end{subfigure}
%\hfill
\begin{subfigure}[t]{0.45\textwidth}
\vspace*{-8.3cm}\includegraphics[width=\textwidth]{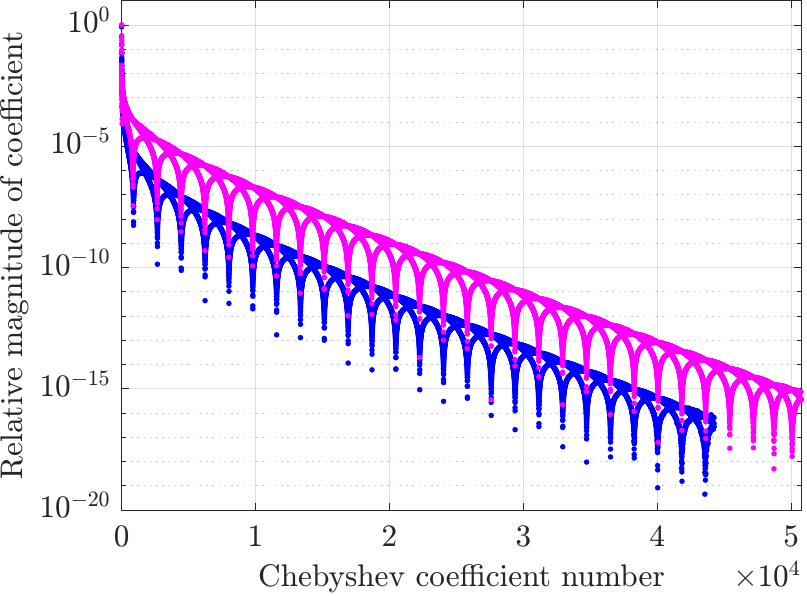}
\caption{$p=1/2$}
\label{fig:cheb_coeffs_p1}
\end{subfigure}
%\vskip\baselineskip
%\hfill
\begin{subfigure}[t]{0.45\textwidth}
\vspace*{-2.2cm}\includegraphics[width=\textwidth]{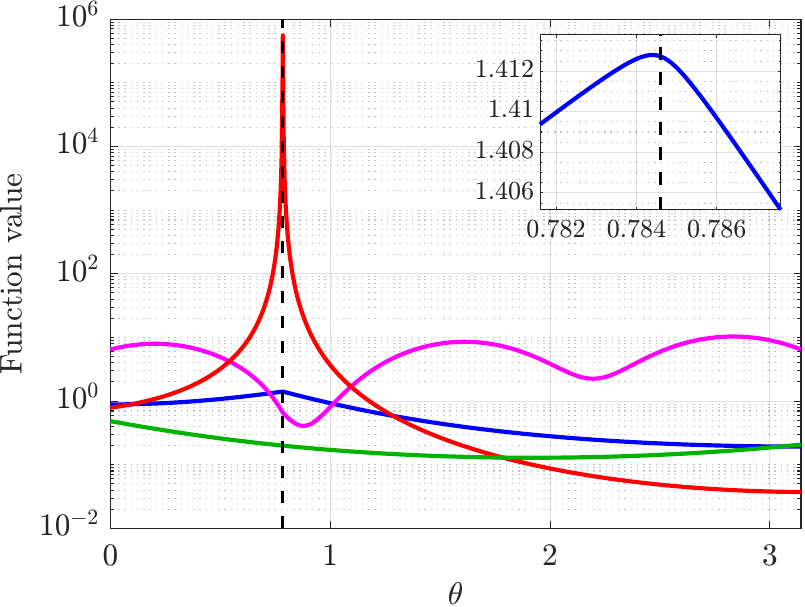}
\caption{$p=3/2$}
\label{fig:cheb_value_p3}
\end{subfigure}
%\hfill
\begin{subfigure}[t]{0.45\textwidth}
\vspace*{-2.2cm}\includegraphics[width=\textwidth]{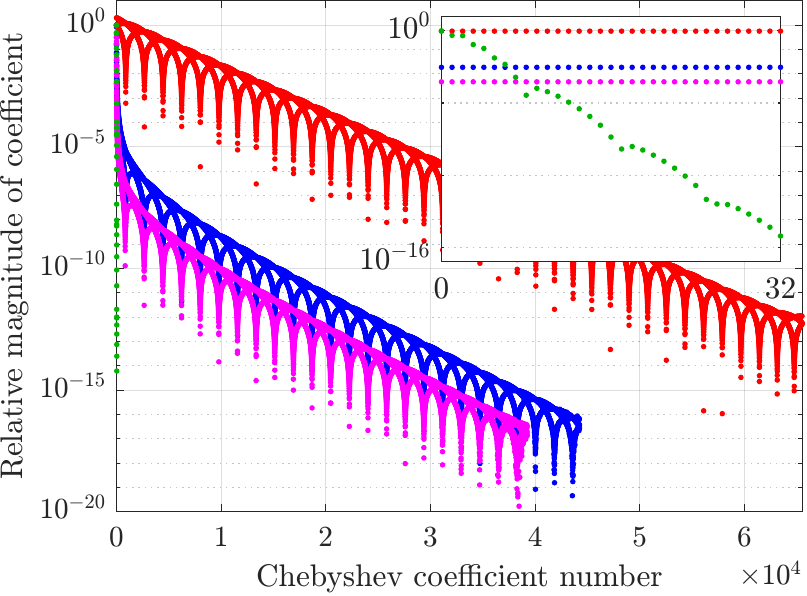}
\caption{$p=3/2$}
\label{fig:cheb_coeffs_p3}
\end{subfigure}
\caption{Typical behavior of the functions appearing in the integrand of the reduced polar integral \eqref{eq:layer_potential_abc}. Panels (a) and (c) show the individual factors for $p=1/2$ and $p=3/2$, respectively; the vertical dashed line marks $\theta=\Re(\thetarootlambda)$ for the given target point. Panels (b) and (d) show the relative decay of the corresponding Chebyshev coefficients. For $p=3/2$, the unregularized factor $(\Rlambda)^{-(p-1/2)}$ exhibits slow coefficient decay, while the regularized quantity $|\theta-\thetaroot^\lambda|^{2p-1}(\Rlambda)^{-(p-1/2)}$ (shown in green) has markedly lower frequency content. Here $\thetarootlambda$ denotes a complex root of $\Rlambda$ defined by \eqref{eq:theta0_spheroid} and \eqref{eq:Rlambda}, respectively. In this example, $k(\xx,\yy)=1$ and $\sigma(\ggamma(\theta,\varphi))=\sin(5\theta)e^{-\cos^2(\varphi)}+1.03$.}
\label{fig:cheb_value_coeffs}
\end{figure}

\subsection{Singularity swap quadrature in the polar variable}\label{ss:ssq_polar}
We now describe the application of singularity swap quadrature in the polar variable for the evaluation of the reduced integral  \eqref{eq:layer_potential_integrated} when $p>1/2$. The treatment is applied locally on a subinterval $[\tha,\thb]\subseteq[0,\pi]$ arising from the adaptive subdivision described later.

The procedure mirrors the steps used in the azimuthal direction in Section \ref{ss:analytic_reduction_to_line_integral}. It consists of identifying the complex roots of the relevant squared-distance function, swapping the singularity to a simpler function, 
expanding the remaining factors in a polynomial basis, and evaluating the resulting nearly singular basis integrals analytically.

Since $\Rlambda(\theta,\xx)$ is real for real $\theta$, its roots come in complex conjugate pairs. For spheroidal surfaces these roots are available analytically, while for general axisymmetric surfaces they are obtained via a one-dimensional complex root-finding procedure; see Section \ref{s:root_finding}. When the generating curve is not excessively curved and the target point $\xx$ is close to it, there is typically only a single pair of roots near the interval under consideration \cite{AFKLINTEBERG2021}. We denote this pair by $\{\thetarootlambda(\xx),\overline{\thetarootlambda(\xx)}\}=\{\theta(\trootlambda(\xx),\tha,\thb),\theta(\overline{\trootlambda(\xx)},\tha,\thb)\}$ where $\theta(t,\tha,\thb)$ is the affine map from $[-1,1]$ to $[\theta_a,\theta_b]$ defined in \eqref{eq:t2theta_map}. For brevity, the dependence of $\trootlambda$ on $\xx$ is suppressed in the remainder of this section.

Following af Klinteberg \& Barnett \cite{AFKLINTEBERG2021}, we define
\begin{equation}
    H(\theta) \coloneqq \frac{1}{\sqrt{a(\theta)}}\Lsq(\theta,\xx)\Fcal(\theta,p,\xx)\frac{|\theta-\thetarootlambda|^{2p-1}}{\left(\Rlambda(\theta,\xx)\right)^{p-1/2}}
    \label{eq:H}
\end{equation}
which allows the integral \eqref{eq:layer_potential_abc} over $[\tha,\thb]$ to be written as
\begin{equation}
u(\xx) = \thsc^{2(1-p)}\int_{-1}^{1} \frac{H(\theta(t,\tha,\thb))}{|t-t_0^\lambda|^{2p-1}}\dt,
\label{eq:Hintegral}
\end{equation}
where $|\theta-\thetarootlambda|=\thsc|t-\trootlambda|$. The singularity swap has transformed the integral on the curve in the polar direction to a denominator that corresponds to that of a straight line.

The function $H(\theta(t,\tha,\thb))$ is expanded in a real monomial basis on $[-1,1]$,
\begin{equation}
    H(\theta(t,\tha,\thb)) \approx \sum_{k=1}^\nt c_kt^{k-1},
    \label{eq:H_poly}
\end{equation}
where the coefficients $c_k$ are obtained by solving a Vandermonde system using the sampled values \linebreak$H(\theta(t_i,\tha,\thb))$ with $t_i$ being the Gauss--Legendre quadrature nodes. Substituting this expansion into the integral yields
\begin{equation}
    u(\xx) \approx \thsc^{2(1-p)}\sum_{k=1}^\nt c_k\nu_k^p(\trootlambda)
    \label{eq:layer_potential_ssq}
\end{equation}
where
\begin{equation}
    \nu_k^p(\trootlambda) = \int_{-1}^1\frac{t^{k-1}}{|t-\trootlambda|^{2p-1}}\dt,\quad k=1,\dots,\nt.
\end{equation}
These basis integrals are evaluated efficiently using the recurrence relations stated in Lemma \ref{lem:nu_rec}, with proof and initial values given in Appendix \ref{a:proofs}.

\begin{lemma}\label{lem:nu_rec}Let $k\in\mathbb{N}$, $p=\overline{p}+3/2$, $\overline{p}\in\mathbb{Z}^+$, and $\trootlambda=t_r+it_i$, where $t_r,t_i\in\mathbb{R}$ and $\trootlambda\notin[-1,1]$. Then, the integrals
\begin{equation}
    \nu_k^p(\trootlambda) = \int_{-1}^1\frac{t^{k-1}}{|t-\trootlambda|^{2p-1}}\dt
\end{equation}
can be expressed as
\begin{equation}
    \nu_k^p(\trootlambda) =
    \begin{cases}
        \dfrac{1-(-1)^{k-2}}{k-2} + 2t_r\nu_{k-1}^p(\trootlambda) - |\trootlambda|^2\nu_{k-2}^p,& p=3/2\text{ and } k\geq 3,\\
        \nu_{k-2}^{p-1}(\trootlambda) + 2t_r\nu_{k-1}^p(\trootlambda) - |\trootlambda|^2\nu_{k-2}^p(\trootlambda),& p>3/2\text{ and } k\geq 3.
    \end{cases}
    \label{eq:nu_rec}
\end{equation}
Moreover, if $t_r=0$, then $\nu_k^p(\trootlambda)=0$ for $k\in 2\mathbb{N}$. The expressions for the initial values for $\nu_k^p(\trootlambda)$ for $p=3/2,~5/2$ are found in \eqref{eq:nu_rec_initial_value_1}-\eqref{eq:nu_rec_initial_value_4} in Appendix \ref{a:proofs}.
\end{lemma}

The accuracy of this quadrature depends on how well $H$ is resolved by polynomial interpolation on the current subinterval. While the singularity swap removes the dominant algebraic singularity associated with $(\Rlambda)^{-(p-1/2)}$, the analytic continuation of $H$ remains limited by branch-point and logarithmic singularities inherited from $\Lsq$ and $\Fcal$. Consequently, a single global polynomial interpolant of $H$ over $[0,\pi]$ is generally inefficient and may converge slowly when the complex singularities lie close to the real interval.

Rather than increasing the polynomial degree, we subdivide the interval $[0,\pi]$ in the polar variable. This subdivision is introduced for efficiency: on each subinterval, the nearest complex roots, when mapped to $[-1,1]$, are sufficiently far from the interval so that $H$ can locally be well approximated by a low-degree polynomial. This enlarges the effective domain of analyticity on each panel and enables accurate polynomial interpolation using a fixed, modest number of Gauss--Legendre nodes (typically 16 or 32). The same subdivision strategy is employed for $p=1/2$, where standard Gauss--Legendre quadrature is used instead of SSQ. Section \ref{ss:local_refinement} describes how this discretization refinement is realized adaptively.

\begin{remark}[Stability and accuracy of the recurrence formulas $\nu_k^p$]\label{rem:nu_rec_stability}When $|\trootlambda|>1$, which can occur when the target point $\xx$ is far from the Gauss--Legendre panel under consideration, errors due to finite precision can be greatly amplified in the forward recurrences \eqref{eq:nu_rec}. Additional loss of accuracy can occur when $\trootlambda$ lies in one of the two cones extending outwards from the endpoints $\pm 1$, as previously noted in \cite{AFKLINTEBERG2021}.

However, such situations are avoided by construction in the adaptive refinement scheme described in Section \ref{ss:local_refinement}. In practice, for $|\trootlambda|\leq1.1$ and moderate values of $n_t$, the forward recurrence errors typically remain small. When $|\trootlambda|>1.1$, the recurrences are instead run backward, assuming that the final two terms can be computed accurately via standard quadrature. An error predictor from \cite{AFKLINTEBERG2022} is used to determine this; if the predicted error is large, the forward recurrence is retained.
\end{remark}

\subsection{Adjoint method for target-specific quadrature weights}\label{ss:adjoint_method}
The evaluation of the azimuthal integral in \eqref{eq:s3q_Iphi} and the polar integral in \eqref{eq:s3q_Itheta} using SSQ relies on the corresponding basis coefficients, namely Fourier coefficients in the azimuthal direction and monomial coefficients in the polar direction. For a fixed target point, these coefficients are determined from samples of a smooth numerator function that includes the kernel $k(\xx,\yy)$ and the layer density $\sigma(\yy)$. In many applications, however, the same target and kernel must be evaluated for several different densities, for instance in iterative boundary integral solvers. In such cases, it is advantageous to precompute target-specific quadrature weights that act directly on the sampled density. We refer to this as the \textit{adjoint method}. Its practical construction in the monomial and Fourier settings is described in \cite[Section 2.2.2]{AFKLINTEBERG2021} and \cite[Section 2.4]{afKlinteberg2024}, respectively.

\subsection{Local refinement scheme}\label{ss:local_refinement}
To efficiently represent the function $H$ in \eqref{eq:H} using a monomial basis, we adopt a local refinement strategy in the polar direction based on panel subdivision. The refinement is driven by how well the function $H$ can be resolved by fixed-order polynomial interpolation, recognizing that its limited regularity is inherited from the branch-point and logarithmic singularities of $\Lsq$ and $\Fcal$, and that these difficulties are localized in the polar variable, as illustrated in Figure \ref{fig:cheb_value_coeffs}.

Let $P\in[-1,1]$ denote the base Gauss--Legendre panel on which both the geometry and layer density are well resolved, and let $\xx$ be a target point with corresponding complex root $\trootlambda$ of $\Rlambda$. We subdivide $P$ into $\npan$ subpanels $[t_a^i,t_b^i]$, $i=1,\dots,\npan$, each equipped with $\nGL$-point Gauss--Legendre nodes. The subdivision is organized as a binary tree: each panel may be bisected into two child panels of equal length, and refinement proceeds recursively.

The refinement is designed to satisfy
\begin{equation}
    \E(\nGL,t_a^i,t_b^i,p,\sigma,\xx) \leq \eps_i,\qquad i=1,\dots,\npan,
    \label{eq:dt_criteria_i}
\end{equation}
where $\E(\nGL, t_\alpha, t_\beta, p, \sigma, \xx)$ denotes the error predictor for a general $\nGL$-point Gauss--Legendre panel on $[t_a, t_b]$, to be discussed in Section \ref{s:error_estimates}, and where the local tolerances satisfy $\sum_{i=1}^{\npan} \eps_i = \eps$, with $\eps$ being the prescribed global error tolerance. Panels for which the error predictor exceeds the assigned tolerance are bisected, and the process continues until the criterion is satisfied on all leaf panels.

This binary-tree structure has important practical advantages. Since all panels at a given refinement level are obtained by repeated bisection of the reference interval $[-1,1]$, interpolation matrices for transferring surface quantities from parent panels to child panels can be precomputed up to a fixed depth. For deeper levels, refinement is handled dynamically using precomputed local interpolation matrices for left and right subpanels. This approach avoids repeated construction of interpolation operators and leads to an efficient and robust implementation.

The output of the algorithm is a subdivision of $[-1,1]$ into $\npan$ $\nGL$-point Gauss--Legendre panels that allow the reduced polar integral \eqref{eq:layer_potential_abc} to be evaluated with accuracy $\eps$. The layer density is interpolated from the quadrature nodes on the base panel $P$ to the nodes on each subpanel using barycentric Lagrange interpolation \cite{barycentric}. On each subpanel, the contribution is then evaluated by Gauss--Legendre quadrature when this is sufficient, and otherwise by the polar SSQ formula \eqref{eq:Hintegral}. The subpanel contributions are finally summed.

To complete the description of the refinement strategy it remains to concretize the error predictor in \eqref{eq:dt_criteria_i}. This is the subject of the following section.

\begin{figure}[!t]
\centering
\begin{tikzpicture}[scale=1]
\draw[thick,|-|] (-5,0) arc[start angle=120,end angle=80,radius=10] pic[pos=0,sloped]{perp_blue} pic[pos=1,sloped]{perp_blue} pic[pos=.25,sloped]{perp_red} pic[pos=.5,sloped]{perp_red} pic[pos=.625,sloped]{perp_red} pic[pos=.6875,sloped]{perp_red} pic[pos=.75,sloped]{perp_red};
\draw[fill, black] (-0.1,1) circle (2pt) node[yshift=-11] {$\mathbf{x}$} node[yshift=25] {};
\node[text width=1cm] at (-3.9,0.8) {$P$};
\end{tikzpicture}
\caption{Illustration of an adaptive panel subdivision.}
\label{fig:local_refinement}
\end{figure}

\section{Error analysis and practical predictors}
\label{s:error_estimates}
The goal of this section is to develop computable quadrature and interpolation error predictors that are sufficiently reliable to drive adaptive panel refinement in the polar direction, rather than to derive sharp or asymptotically exact bounds. Our starting point is a complex-analytic framework, reviewed in Section \ref{ss:quaderr_complex_analysis}, in which the quadrature and interpolation errors admit exact contour-integral representations. Building on this framework and prior analyses \cite{SORGENTONE2023, AFKLINTEBERG2022}, we obtain practical predictor formulas by introducing asymptotic simplifications and retaining only the nearest complex singularities. These predictors, derived in Sections \ref{ss:error_estimates_quad}--\ref{ss:interp_error_estimates}, form the basis of the adaptive panel subdivision strategy described in Section \ref{ss:local_refinement}.

\subsection{Quadrature and interpolation error formulas using complex analysis}\label{ss:quaderr_complex_analysis}
Let $g$ be a function defined on a real interval $E\subset\mathbb{R}$, and consider the integral
\begin{equation}
    \I[g] = \int_E g(t)\dt,
    \label{eq:Ig}
\end{equation}
approximated by an $n$-point quadrature rule with quadrature nodes $\{t_\ell\}_{\ell=1}^n$ and corresponding weights $\{w_\ell\}_{\ell=1}^n$,
\begin{equation}
    \Q_n[g] = \sum_{\ell=1}^n g(t_\ell)w_\ell.
    \label{eq:Qng}
\end{equation}
The quadrature error is then given by
\begin{equation}
    \E_n[g] = \I[g] - \Q_n[g].
    \label{eq:E}
\end{equation}
The rate at which $\E_n[g]$ decay as $n$ increases depends on the smoothness of $g$ and the chosen quadrature rule.

Classical estimates of $\E_n[g]$ for Gauss--Legendre quadrature typically involve high-order derivatives of $g$. Such estimates are therefore only practical for very smooth integrands and tend to overestimate the error when $g$ has singularities near $E$. In nearly singular settings, they may even incorrectly suggest an increase in error as $n$ grows, despite the actual error decreasing \cite{Klinteberg2017}.

A more informative approach in this regime is based on complex analysis, as described by Donaldson \& Elliot \cite{DONALDSON1972}. They express the error as a contour integral in the complex plane,
\begin{equation}
    \E_n[g] = \frac{1}{2\pi i} \int_C g(t)k_n(t)\dt,
    \label{eq:Ecomplex}
\end{equation}
where the contour $C$ in Figure \ref{fig:contour}, containing the integration interval $E$, is chosen so that the complex continuation of $g$ is analytic on and inside $C$. The error (or remainder) function $k_n(t)$ depends on the quadrature rule and is detailed in Section \ref{ss:gauss_legendre} for the Gauss--Legendre quadrature rule.

A central technique in deriving asymptotic estimates of \eqref{eq:Ecomplex}, introduced in \cite{ELLIOT2008}, involves deforming the contour $C$ away from $E$ while avoiding all singularities. Assuming the integrand tends to zero faster than $|t|^{-1}$ as $|t|\rightarrow\infty$, contributions from parts of $C$ far from $E$ vanish. Thus, $C$ can be extended to infinity, leaving only the contribution from the singularities and their branch cuts. The particular choice of branch cuts in \cite{ELLIOT2008} is essential for obtaining accurate, simplified analytical asymptotic estimates. For functions with multiple pairs of complex conjugate singularities, such as \eqref{eq:generic_layer_potential}, all pairs contribute in principle, but typically only the pair closest to $E$ has a significant contribution, and the others can safely be discarded \cite{AFKLINTEBERG2022,AFKLINTEBERG2021}. \footnote{For Gauss--Legendre quadrature, singularities farther from $E=[-1,1]$ are exponentially suppressed: if $g$ extends analytically to a Bernstein ellipse of radius $\varrho$, then $\E_n[g]=\mathcal{O}(\varrho^{-2n})$ \cite{trefethen2019}. For example, with $n=16$ and a singularity at $z=0.5i$, a second singularity at $2z$ contributes about $10^{-6}$ times as much to the error. If two singularities are located at comparable distances from $E$ and only one is retained in the predictor, the neglected contribution is typically of the same order of magnitude, resulting in a difference by at most a modest constant factor.} The functions considered in this work exhibit branch-point singularities, meaning the main error contribution comes from neighborhoods of these branch points and their branch cuts, illustrated in Figure \ref{fig:contour}.

\begin{figure}[!t]
    \centering
    \begin{tikzpicture}[scale=1.15]

% dati numerici
\tikzmath{
    % dati cerchietto primo quadrante
    \Cx = 1; \Cy = 0.4; \RR = .15;
    %
    % angoli speciali cerchietto (in gradi)
    \ca1 = 42; % inizio buco
    \ca2 = 105; % fine buco
    %
    % angoli speciali ellisse (in gradi)
    \ea1 = 30; \succea1 = \ea1 + 1;
    \ea2 = 50; \succea2 = \ea2 + 1;
    \ea3 = 57; \succea3 = \ea3 + 1;
    \ea4 = 75; \succea4 = \ea4 + 1;
    \ea5 = 135; \succea5 = \ea5 + 1;
    \ea6 = 360-\ea5; \succea6 = \ea6 + 1;
    \ea7 = 360-\ea4; \succea7 = \ea7 + 1;
    \ea8 = 360-\ea3; \succea8 = \ea6 + 1;
    \ea9 = 360-\ea2; \succea9 = \ea7 + 1;
    \ea0 = 360-\ea1; \succea0 = \ea0 + 1;
}

% assi cartesiani
\draw[thick,->] (-3.5,0) -- (3.5,0);
\draw[thick,->] (0,-2.5) -- (0,2.5);

% ellisse guida
% \draw[dotted] (0,0) ellipse (3 and 2);

% definisci tutti i punti dell'ellisse
% (P1), (P2), ..., (P360) corrispondenti ai gradi
\foreach \a in {0, ..., 360} {
    \tikzmath {
        \x = 3*cos(\a);
        \y = 2*sin(\a);
    }
    \coordinate (P\a) at (\x,\y);
}

% cerchietti guida
%\draw[dotted] (\Cx,\Cy) circle (\RR);
%\draw[dotted] (\Cx,-\Cy) circle (\RR);

% definisci tutti i punti del cerchietto sopra
% (Q1), (Q2), ..., (Q360) corrispondenti ai gradi
\foreach \a in {0, ..., 360} {
    \tikzmath {
        \x = \Cx + \RR*cos(\a);
        \y = \Cy + \RR*sin(\a);
    }
    \coordinate (Q\a) at (\x,\y);
}

% definisci tutti i punti del cerchietto sotto
% (R1), (R2), ..., (R360) corrispondenti ai gradi
\foreach \a in {0, ..., 360} {
    \tikzmath {
        \x = \Cx + \RR*cos(\a);
        \y = -\Cy - \RR*sin(\a);
    }
    \coordinate (R\a) at (\x,\y);
}

% pallino magenta superiore
\draw[fill, magenta] (\Cx,\Cy) circle (0.05) node[left, xshift=-5pt] {$t_0$};

% curva magenta superiore
\coordinate (ext1) at ($(P53)+(.5,.4)$);

\coordinate (CC1) at 
($(\Cx,\Cy)!.3!(ext1)+(-.2,.2)$);

\coordinate (CC2) at 
($(\Cx,\Cy)!.6!(ext1)+(-.2,.2)$);

\draw[thick, magenta] 
(\Cx,\Cy) .. controls (CC1) and (CC2) .. (ext1);

\draw[magenta] (ext1) node[right] {$B(t_0)$};

% pallino magenta inferiore
\draw[fill, magenta] (\Cx,-\Cy) circle (0.05)
node[left, xshift=-5pt] {$\overline{t_0}$};

% curva magenta inferiore
\coordinate (ext2) at ($(P307)+(.5,-.4)$);

\coordinate (CC3) at 
($(\Cx,-\Cy)!.3!(ext2)+(-.2,-.2)$);

\coordinate (CC4) at 
($(\Cx,-\Cy)!.6!(ext2)+(-.2,-.2)$);

\draw[magenta] (ext2) node[right] {$B(\overline{t_0})$};

\draw[thick, magenta] 
(\Cx,-\Cy) .. controls (CC3) and (CC4) .. (ext2);

% INIZIO PERCORSO BLU

% primo trato di ellisse
\draw[blue] (P0)
\foreach \a in {1, ..., \ea1} {
    -- (P\a)
};

% prima freccetta
\draw[thick,blue, ->] (P\ea1) -- (P\succea1);

% secondo tratto di ellisse
\draw[blue] (P\ea1)
\foreach \a in {\succea1, ..., \ea2} {
    -- (P\a)
};

% C_1^-
\coordinate (C1) at ($(P\ea2)!.3!(Q\ca1)+(-.1,.1)$);

\coordinate (C2) at ($(P\ea2)!.6!(Q\ca1)+(-.1,.1)$);

\tikzmath{
    \precea2 = \ea2 -1;
    \precca1 = \ca1 - 1;
}

\draw[blue] (P\precea2) -- (P\ea2) 
.. controls (C1) and (C2) .. (Q\ca1) -- (Q\precca1);

\draw[thick, blue, -{>[scale=.8]}] ($(C2)+(.05,0)$) -- +(-.02,-.03) node[right, yshift=-4] {$C_1^-$};
%node[right, yshift=-4] {$C_1^-$};

% cerchietto blu
\draw[blue] (Q\ca1)
\foreach \a in 
{\ca1, ..., 0, 360, 359, ..., \ca2} {
    -- (Q\a)
};

% C_1*ì
\coordinate (C3) at ($(Q\ca2)!.3!(P\ea3)+(-.1,.1)$);

\coordinate (C4) at ($(Q\ca2)!.6!(P\ea3)+(-.1,.1)$);

\tikzmath{
    \succea3 = \ea3 + 1;
    \succca2 = \ca2 + 1;
}

\draw[blue] (Q\succca2) -- (Q\ca2) 
.. controls (C3) and (C4) .. (P\ea3)
-- (P\succea3);

\draw[thick, blue, -{>[scale=.8]}] ($(C3)+(.05,0)$) -- +(.02,.03) node[left, xshift = .5, yshift=.4] {$C_1^+$};
%node[left, xshift = .5, yshift=.4] {$C_1^+$};

% nuovo tratto di ellisse
\draw[blue] (P\ea3)
\foreach \a in {\ea3, ..., \ea4} {
    -- (P\a)
};

\draw[thick,blue, ->] (P\ea4) -- (P\succea4);

\draw[blue] (P\ea4)
\foreach \a in {\ea4, ..., \ea5} {
    -- (P\a)
};

\draw[thick, blue, ->] (P\ea5) -- (P\succea5) 
node [above left, xshift=2pt] {$C$};

\draw[blue] (P\ea5)
\foreach \a in {\ea5, ..., \ea6} {
    -- (P\a)
};

\draw[thick,blue, ->] (P\ea6) -- (P\succea6);

\draw[blue] (P\ea6)
\foreach \a in {\ea6, ..., \ea7} {
    -- (P\a)
};

\draw[thick,blue, ->] (P\ea7) -- (P\succea7);

% C_2^-
\coordinate (C5) at 
($(R\ca2)!.3!(P\ea8)+(-.1,-.1)$);

\coordinate (C6) at ($(R\ca2)!.6!(P\ea8)+(-.1,-.1)$);

\tikzmath{
    \precea8 = \ea8 - 1;
    \succca2 = \ca2 + 1;
}

\draw[blue] (R\succca2) -- (R\ca2)
.. controls (C5) and (C6) .. (P\ea8)
-- (P\precea8);

\draw[thick, blue, -{>[scale=.8]}] 
($(C5)+(.05,0)$) -- +(-.02,+.03) node[left, xshift = .5, yshift=-.4] {$C_2^-$};

% cerchietto blu
\draw[blue] (R\ca1)
\foreach \a in 
{\ca1, ..., 0, 360, 359, ..., \ca2} {
    -- (R\a)
};

% C_2^-
\coordinate (C7) at ($(P\ea9)!.3!(R\ca1)+(-.1,-.1)$);

\coordinate (C8) at
($(P\ea9)!.6!(R\ca1)+(-.1,-.1)$);

\tikzmath{
    \succea9 = \ea9 + 1;
    \precca1 = \ca1 - 1;
}

\draw[blue] (P\succea9) -- (P\ea9) 
.. controls (C7) and (C8) .. (R\ca1)
-- (R\precca1);

\draw[thick, blue, -{>[scale=.8]}] ($(C8)+(.05,0)$) -- +(.02,-.03) node[right, yshift=4] {$C_2^+$};

\draw[blue] (P\ea7)
\foreach \a in {\ea7, ..., \ea8} {
    -- (P\a)
};

\draw[blue] (P\ea9)
\foreach \a in {\ea9, ..., \ea0} {
    -- (P\a)
};

\draw[thick,blue, ->] (P\ea0) -- (P\succea0);

\draw[blue] (P\ea0)
\foreach \a in {\ea0, ..., 360} {
    -- (P\a)
};

% segmento rosso
\draw[thick, red] (-2,0)  -- (2,0);

% pallocchi rossi
\draw[fill, red] (-2,0) circle (2pt);
\draw[fill, red] (2,0) circle (2pt);

\draw[red] (-1,0) node[above] {$E$};

\draw[fill, black] (-0.5,0) circle (1.5pt);
\draw[black] (-0.5,-0.06) node[below] {$\tau$};

% etichetta rossa
%\draw[red] (-1,0) node[below] {$E$};
\end{tikzpicture}
    \caption{Sketch of the contour $C$, the base interval $E$, interpolation point $\tau$, and the deformations $C_1=C_1^-\cup C_1^+$ and $C_2=C_2^-\cup C_2^+$ circumventing the singularities $\troot$ and $\overline{\troot}$ with branch cuts $B(\troot)$ and $B(\overline{\troot})$, respectively.}
    \label{fig:contour}
\end{figure}

The same complex analysis framework applies also to polynomial interpolation. Given $n$ distinct sample points $\{t_k\}_{k=1}^n\subset E$, then there exists a unique polynomial $p_{n-1}$ of degree at most $n-1$ such that $p(t_k)=g(t_k)$, $k=1,\dots,n$. This polynomial can be represented in various ways, but due to its superior stability and efficiency, we construct it using the barycentric Lagrange interpolation formula \cite{barycentric,trefethen2019}
\begin{equation}
    p_{n-1}(t) = \elln(t)\sum_{k=1}^n\frac{w_kg(t_k)}{t-t_k},
\end{equation}
where the barycentric weights are
\begin{equation}
    w_k = \frac{1}{\elln^{\prime}(t_k)} = \left(\prod_{j=1,~j\neq k}^n(t_k-t_j)\right)^{-1},\quad k=1,\dots,n
    \label{eq:barycentric_weights}
\end{equation}
and the node polynomial is
\begin{equation}
    \elln(t) = \prod_{k=1}^n (t-t_k).
    \label{eq:elln}
\end{equation}
The polynomial interpolation error at $\tau\in E$,
\begin{equation}
    \EI_n[g](\tau) = g(\tau)-p_{n-1}(\tau),
\end{equation}
admits, by the Hermite interpolation formula \cite[Theorem 11.1]{trefethen2019}, the contour representation
\begin{equation}
    \EI_n[g](\tau) = \frac{1}{2\pi i} \int_C \frac{\elln(\tau)}{\elln(t)}\frac{g(t)}{(t-\tau)}\dt.
    \label{eq:Einterp_complex}
\end{equation}
This formula holds under the same assumptions as stated below \eqref{eq:Ecomplex}, namely that $g$ is analytic on and inside $C$, and that $E$ is contained within $C$. We will be concerned with the maximum absolute interpolation error over $E$,
\begin{equation}
    \left|\EI_n[g]\right| = \max_{\tau\in E}~\left|\EI_n[g](\tau)\right|.
    \label{eq:Einterp_complex_max}
\end{equation}
Comparing \eqref{eq:Einterp_complex} with \eqref{eq:Ecomplex}, we see that, the error function $k_n(t)$ is replaced in the interpolation case by the factor $\elln(\tau)/(\elln(t)(t-\tau))$, which depends on the node distribution and evaluation point $\tau$. 
Crucially, both factors share the property that their magnitude decays rapidly away from the real interval $E$.

\subsection{Quadrature error predictors for adaptivity in polar direction}\label{ss:error_estimates_quad}
We now construct a quadrature error predictor for the function $\Lsq$, defined in \eqref{eq:Lsq}, repeated here for convenience,
\begin{equation}
    \Lsq(\theta,\xx) = \left(\lamb+\sqrt{\lamb^2-\rho^2\sin^2(\theta)}\right)^{-1/2},
    \label{eq:Lsq_est}
\end{equation}
with $\lambda$ defined in \eqref{eq:lambda}.
This predictor is used to guide adaptive subdivision of the polar interval in the case $p=1/2$, where no singularity-swap formulation is applied. 

We begin by deriving a predictor applicable to a general $n$-point quadrature rule, expressed in terms of its error function $k_n(t)$, and then specialize the construction to the Gauss–Legendre rule. The resulting predictor is summarized in \textit{Error predictor \ref{est:B_quad}}. 

\subsubsection{General results}\label{ss:quad_general_results}
Consider the integral
\begin{equation}
    \int_{\tha}^{\thb} \Lsq(\theta,\xx)\dtheta = \int_E \Lsq(\theta(t,\tha,\thb),\xx)\thsc\dt,
\label{eq:Bintegral}
\end{equation}
where $E=[-1,1]$, $\theta(t,\tha,\thb)$ is the linear map defined in \eqref{eq:t2theta_map}, and $\thsc$ is the associated scaling factor. We denote the quadrature error of this integral, evaluated with an $n$-point rule by $\E[\Lsq](\xx,n,\tha,\thb)$.
Let $\{\trootlambda,\overline{\trootlambda}\}$ be the pair of branch points of the integrand closest to $E$, defined as the roots of $\Rlambda(\theta(t,\tha,\thb),\xx)$ in \eqref{eq:Rlambda}. For notational convenience, we write $\theta(t,\tha,\thb)=\thetat$ and denote these roots by $\{\troot,\overline{\troot}\}$ for the remainder of Section \ref{s:error_estimates}.

From the contour representation reviewed in Section \ref{ss:quaderr_complex_analysis}, the quadrature error admits an exact contour integral expression. By deforming the contour away from the real interval and retaining only the contributions associated with the nearest branch points, the error is modeled by the sum of two branch-cut integrals,
\begin{equation}
    \E[\Lsq](\xx,n,\tha,\thb) \approx \frac{1}{2\pi i}\int_{C_1}\Lsq(\thetat,\xx)k_n(t)\thsc\dt + \frac{1}{2\pi i}\int_{C_2}\Lsq(\thetat,\xx)k_n(t)\thsc\dt\eqqcolon \E_1+\E_2,
\end{equation}
where $C_1$ and $C_2$ are the curves going around $\troot$ and $\overline{\troot}$, respectively (see Figure \ref{fig:contour}).

We focus on the contribution $\E_1$ from the curve $C_1$.
The integrand contains square-root branch-point singularities at $\troot$ and $\overline{\troot}$. By isolating these factors, $\E_1$ can be written as
\begin{equation}
    \E_1 = \frac{\thsc}{2\pi i}\int_{C_1} \frac{k_n(t)}{\left(\lambda(\thetat)+\Gbar(\thetat)^{1/2}\thsc(t-\troot)^{1/2}(t-\overline{\troot})^{1/2}\right)^{1/2}}\dt,
    \label{eq:Lsq_quad_E1}
\end{equation}
where
\begin{equation}
    \Gbar(\theta) \coloneqq \frac{\lambda(\theta)^2-\rho^2\sin^2(\theta)}{(\theta-\thetaroot)(\theta-\overline{\thetaroot})}.
    \label{eq:Gbar}
\end{equation}
Let $C_1^+$ denote the side of $C_1$ going outwards, from $\troot$ to $\infty$, and let $C_1^-$ denote the other side going in the opposite direction such that $C_1=C_1^-\cup C_1^+$, illustrated in Figure \ref{fig:contour}. Defining the jump across the branch cut as
\begin{equation}
    (t-\troot)\Big.\Big\vert_{C_1^+} = (t-\troot)\Big.\Big\vert_{C_1^-}e^{-2\pi i},
\end{equation}
we obtain
\begin{equation}
    (t-\troot)^{1/2}\Big.\Big\vert_{C_1^+} = -(t-\troot)^{1/2}\Big.\Big\vert_{C_1^-}.
\end{equation}
This allows $\E_1$ to be written as the difference of two integrals: an integral from $t_0$ to $\infty$ with $(t-\troot)^{1/2}\Big.\Big\vert_{C_1^+}$, and another from $\infty$ to $t_0$ with $(t-\troot)^{1/2}\Big.\Big\vert_{C_1^-}$. With the definition
\begin{equation}
    J_{\pm}(\troot,n) \coloneqq \int_{\troot}^\infty \frac{k_n(t)}{\left(\lambda(\thetat)\pm\Gbar(\thetat)^{1/2}\thsc(t-\troot)^{1/2}(t-\overline{\troot})^{1/2}\right)^{1/2}}\dt,
    \label{eq:Jpm}
\end{equation}
we obtain
\begin{equation}
    \E_1 = \frac{\thsc}{2\pi i}\left(J_+(\troot,n) - J_-(\troot,n)\right) \eqqcolon \frac{\thsc}{2\pi i}J(\troot,n)
    \label{eq:J}
\end{equation}
The contribution $\E_2$ from $C_2$ for the branch point $\trootbar$ is computed analogously and satisfies $\E_2=\overline{\E_1}$. Combining the two contributions, discarding the real part, and taking the absolute value, yields the following general quadrature error predictor
\begin{equation}
    \left|\E[\Lsq](\xx,n,\tha,\thb)\right| \approx \frac{\thsc}{\pi}\left|J(\troot,n)\right|.
    \label{eq:general_quad_est_B}
\end{equation}
To render this predictor computationally useful, an efficient and sufficiently accurate way to evaluate the integrals defining $J_\pm$ is needed. We now consider that problem in the case of the Gauss--Legendre rule.

\subsubsection{Gauss--Legendre rule}\label{ss:gauss_legendre}
For Gauss--Legendre quadrature on $E=[-1,1]$, the error function $k_n$ does not have a closed form expression, but was shown in \cite{DONALDSON1972,ELLIOT2008} to satisfy the asymptotic formula
\begin{equation}
    k_n(z) \simeq \frac{2\pi}{\xi(z)^{2n+1}}
    \label{eq:GL_knz}
\end{equation}
as $n\rightarrow\infty$ with $\xi(z) = z + \sqrt{z^2-1}$. Here, and for the remainder of this paper, we shall use the symbol $\simeq$ to mean ``asymptotically equal to''. That is, $a(n)\simeq b(n)$, for large $n$, if $\lim_{n\rightarrow\infty}a(n)/b(n)=1$. Moreover, we will define $\sqrt{z^2-1}$ as $\sqrt{z+1}\sqrt{z-1}$ with $-\pi<\arg(z\pm1)\leq\pi$.

The integrands defining \eqref{eq:Jpm} contain a branch point at $\troot$. Following \cite{ELLIOT2008}, we define a branch cut extending from $\troot$ to $\infty$ that does not intersect $[-1,1]$,
\begin{equation}
    B(\troot) = \left\{\branchparam(s)\in\mathbb{C} : \branchparam(s) = \frac{1}{2}\left(\zeta(s)+\frac{1}{\zeta(s)}\right),\quad 1\leq s<\infty\right\},
    \label{eq:branch_cut}
\end{equation}
where
\begin{equation}
    \zeta(s) = \zeta_0s,\quad \zeta_0 = \troot + \sqrt{\troot^2-1}.
    \label{eq:branch_cut_extra}
\end{equation}
One may verify that, as intended, $\branchparam(1)=\troot$. This particular parametrization leads to the following helpful simplification of $k_n$,
\begin{equation}
    k_n(\branchparam(s))\simeq \frac{2\pi}{\zeta(s)^{2n+1}},
    \label{eq:kn_branch_cut}
\end{equation}
where
\begin{equation}
    \bernsteinr(\troot)=|\zeta_0|
    \label{eq:bernstein_radius}
\end{equation}
is the so-called Bernstein radius of $\troot$. 

Substituting this asymptotic form into the expressions for $J_\pm$, we obtain the following model for the branch-cut contributions
\begin{equation}
    J_\pm(\troot,n) \simeq \frac{2\pi}{\zeta_0^{2n+1}}\int_1^\infty \frac{1}{s^{2n+1}} \frac{\branchparamprime(s)\ds}{\left(\lambda(\theta(\branchparam(s))\pm\Gbar(\theta(\branchparam(s)))^{1/2}\thsc(\branchparam(s)-\troot)^{1/2}(\branchparam(s)-\overline{\troot})^{1/2}\right)^{1/2}}.
\label{eq:Jbranch}
\end{equation}
Inserting this expression into the general predictor derived in the previous subsection leads directly to the Gauss--Legendre quadrature error predictor stated below. 

\begin{errest}[Gauss--Legendre quadrature, $\Lsq$]\label{est:B_quad}Let $\xx=(x,y,z)$ be a target point with $\rho=\sqrt{x^2+y^2}$, and let $\gamma_\rho(\theta)=(a(\theta)\sin(\theta),b(\theta)\cos(\theta))$ parametrize a smooth curve in the $\rho z$-plane over $[\tha,\thb]$.
Denote by $\{\troot,\trootbar\}$ the pair of complex roots of $\Rlambda(\theta(t,\tha,\thb)),\xx)$ in \eqref{eq:Rlambda} closest to $E=[-1,1]$. The functions 
$a(\theta)$ and $b(\theta)$
are introduced in \eqref{eq:gamma_axi}.

The Gauss--Legendre quadrature error for evaluating \eqref{eq:Bintegral} with an $n$-point rule is predicted by
\begin{equation}
    \left|\EQ[\Lsq](\xx,n,\tha,\thb)\right| \approx \alphaQ(2n+1)\betaQ(2n+1),
    \label{eq:Lsq_quad_est}
\end{equation}
where
\begin{equation}
    \alphaQ(m) = \frac{2}{\varrho(t_0)^{m}}\thsc,\qquad \betaQ(m) = \left|\UQ_+(m)-\UQ_-(m)\right|,
    \label{eq:betaQ}
\end{equation}
and
\begin{equation}
    \UQ_\pm(m) = \int_{1}^\infty\frac{1}{s^{m}} \frac{\branchparamprime(s)\ds}{\left(\lambda(\theta(\branchparam(s))\pm\Gbar(\theta(\branchparam(s)))^{1/2}\thsc(\branchparam(s)-\troot)^{1/2}(\branchparam(s)-\overline{\troot})^{1/2}\right)^{1/2}}.
    \label{eq:UQ_pm}
\end{equation}
Here, the functions $\Gbar$, $\branchparam$, and $\varrho$ are defined in \eqref{eq:Gbar}, \eqref{eq:branch_cut}, and \eqref{eq:bernstein_radius}, respectively.
\end{errest}

\subsection{Interpolation error predictor for adaptivity in polar direction}\label{ss:interp_error_estimates}
We now construct an interpolation error predictor for polynomial interpolation of the function $\Lsq$ defined in \eqref{eq:Lsq_est}. As in the quadrature case, the purpose is not to obtain a sharp upper bound, but a reliable efficiently computed quantity suitable for driving adaptive panel refinement. The interpolation is performed on the base interval $E=[-1,1]$, with the linear map $\theta(t,\tha,\thb)$ defined in \eqref{eq:t2theta_map}, using the $n$ Gauss--Legendre nodes.

Starting from the complex contour representation of the interpolation error reviewed in Section \ref{ss:quaderr_complex_analysis}, and using the same notation as in Section \ref{ss:error_estimates_quad}, the pointwise interpolation error at $\tau\in E$ admits the representation
\begin{equation}
    \left|\EI_n[\Lsq](\xx,\tau)\right| \approx \frac{1}{\pi}\left|\elln(\tau)\left(\JI_+(\troot,n)-\JI_-(\troot,n)\right)\right|,
    \label{eq:EI}
\end{equation}
where the branch-cut integrals $\JI_\pm$ are defined as
\begin{equation}
    \JI_\pm(\troot,n) \approx \int_1^\infty\frac{1}{(\branchparam(s)-\tau)\elln(\branchparam(s))}\frac{\branchparamprime(s)\ds}{\left(\lambda(\theta(\branchparam(s))+\Gbar(\theta(\branchparam(s)))^{1/2}\thsc(\branchparam(s)-\troot)^{1/2}(\branchparam(s)-\overline{\troot})^{1/2}\right)^{1/2}}.
    \label{eq:JIbranch}
\end{equation}
The main difference between \eqref{eq:Jbranch} and \eqref{eq:JIbranch} is the factor $1/\elln(\branchparam(s))$, which encodes the dependence on the interpolation nodes. For Gauss--Legendre nodes, this factor can be characterized explicitly.

\begin{lemma}[Asymptotics of the Gauss--Legendre node polynomial]
Let $\{t_k\}_{k=1}^n$ denote the Gauss--Legendre nodes on $[-1,1]$, and let $\ell_n(t) = \prod_{k=1}^n (t-t_k)$ be the node polynomial. Then as $n\rightarrow\infty$, for any complex $t\notin[-1,1]$,
\begin{equation}
    \ell_n(t) \simeq \frac{2^n(n!)^2}{(2n)!(2\pi n)^{1/2}}\frac{1}{(t^2-1)^{1/4}}\left(t+\sqrt{t^2-1}\right)^{n+1/2}.
    \label{eq:node_poly_asymptotic}
\end{equation}
\label{lem:node_poly_asymptotic}
\end{lemma}

\begin{proof}
The Gauss--Legendre nodes are by definition the zeros of the Legendre polynomial $P_n(t)$. Since both $P_n$ and $\ell_n$ are degree-$n$ polynomials with identical zeros, they differ only by a constant factor $\ell_n(t)=c_nP_n(t)$. To determine $c_n$, we compare leading coefficients. The leading coefficient of $P_n$ is $(2n)!/(2^n(n!)^2)$. Since $\ell_n$ is monic, its leading coefficient is $1$. Therefore, $c_n=(2^n(n!)^2)/(2n)!$.

We now invoke Theorem 8.2.1 of \cite{szego}, which states that for complex $t\notin[-1,1]$,
\begin{equation}
    P_n(t) \simeq \frac{1}{(2\pi n)^{1/2}}\frac{1}{(t^2-1)^{1/4}}\left(t+\sqrt{t^2-1}\right)^{n+1/2},
\end{equation}
as $n\rightarrow\infty$. Substituting this into the relation between $\ell_n$ and $P_n$ above yields \eqref{eq:node_poly_asymptotic}.
\end{proof}

Along the branch cut $\branchparam$ \eqref{eq:branch_cut} associated with the singularity $t_0$, the right-most term in \eqref{eq:node_poly_asymptotic} simplifies to $(\zeta_0s)^{n+1/2}$, $1\leq s<\infty$. Thus the interpolation error inherits an exponential factor $|\zeta_0|=\varrho(t_0)^{-(n+1/2)}$, recalling \eqref{eq:bernstein_radius}, in contrast to the quadrature case, where the decay is governed by $\varrho(t_0)^{-(2n+1)}$. 

By combining Lemma \ref{lem:node_poly_asymptotic} with \eqref{eq:JIbranch} and maximizing over $\tau\in\{-1,0,1\}$, we obtain the final interpolation error predictor in \eqref{eq:Lsq_interp_est}. This maximization is supported by numerical observations.

\begin{errest}[Polynomial interpolation at Gauss--Legendre nodes, $\Lsq$]\label{est:Binterp}Consider the function $\Lsq$ defined by \eqref{eq:Lsq_est}. Let all quantities be defined as in \textit{Error predictor \ref{est:B_quad}}. Then the maximum interpolation error over the panel may be predicted by
\begin{equation}
    \left|\EI_n[\Lsq](\xx,n,\tha,\thb)\right| \approx \alphaI(n) \max_{\tau\in\{-1,0,1\}}\betaI(n,\tau),
    \label{eq:Lsq_interp_est}
\end{equation}
where
\begin{equation}
    \alphaI(m) = \frac{1}{\varrho(t_0)^{m+1/2}}\frac{(2m)!(2m)^{1/2}}{2^m(m!)^2\pi^{1/2}},\qquad \betaI(m,\tau) = \left|\ell_m(\tau)\left(\UI_+(m,\tau) - \UI_-(m,\tau)\right)\right|,
    \label{eq:betaI}
\end{equation}
with
\small
\begin{equation}
    \UI_\pm(m,\tau) = \int_{1}^\infty\frac{1}{s^{m+1/2}} \frac{1}{(\branchparam(s)-\tau)}\frac{1}{(\branchparam(s)^2-\tau)^{1/4}}\frac{\branchparamprime(s)\ds}{\left(\lambda(\theta(\branchparam(s))\pm\Gbar(\theta(\branchparam(s)))^{1/2}\thsc(\branchparam(s)-\troot)^{1/2}(\branchparam(s)-\overline{\troot})^{1/2}\right)^{1/2}},
    \label{eq:UI_pm}
\end{equation}\normalsize
where $\ell_m$ is the $m$-point node polynomial defined in \eqref{eq:elln}.
\end{errest}

\subsection{Numerical evaluation of branch-cut integrals}\label{ss:numerical_eval_branch_cut_int}
Both the quadrature and interpolation error predictors reduce to branch-cut integrals $\betaQ$ and $\betaI$ in \eqref{eq:betaQ} and \eqref{eq:betaI}, respectively. These integrals take the form
\begin{equation}
    \int_1^\infty s^{-m}F(s)\ds,
\end{equation}
with $m=2n+1$ in the quadrature case and $m=n+1/2$ in the interpolation case. The factor $s^{-m}$ introduce rapid algebraic decay also for moderate $n$, localizing the integral near $s=1$.

We therefore truncate the interval $[1,\infty)$ to $[1,L]$, where $L$ is chosen such that $s^{-m}=\kappa$ at $s=L$. In all experiments, we set $\kappa=10^{-10}$. The remaining factor $F(s)$ is found to be surprisingly smooth for all target point locations. The truncated integral is therefore efficiently evaluating using Gauss--Legendre quadrature on $[1,L]$, with eighth nodes typically being sufficient for both predictors.

The resulting computational cost is negligible compared to the surface quadrature itself, ensuring that the adaptive refinement remains efficient.

\subsection{Performance of error predictors for adaptivity}
We now assess the performance of the derived quadrature and interpolation error predictors for the function $\Lsq$ in \eqref{eq:Lsq_est}. The aim is to verify that the predictors correctly capture the dependence on target point location, reproduce the expected decay with increasing $n$, and remain reliable in the close-evaluation regime.

Figure \ref{fig:errest_B_quad_ip} compares measured and predicted errors for a fixed value of $n$, for both a convex spheroid and a non-convex ``smooth star'' geometry. The results show that the predictors perform well for both geometries across all target point locations, including those close to the symmetry axis.

\begin{figure}[!t]
\centering
\begin{subfigure}[t]{0.45\textwidth}
%[trim={left bottom right top},clip]
\includegraphics[trim={1.2cm 0.0cm 1.1cm 0.6cm},clip,width=0.975\textwidth]{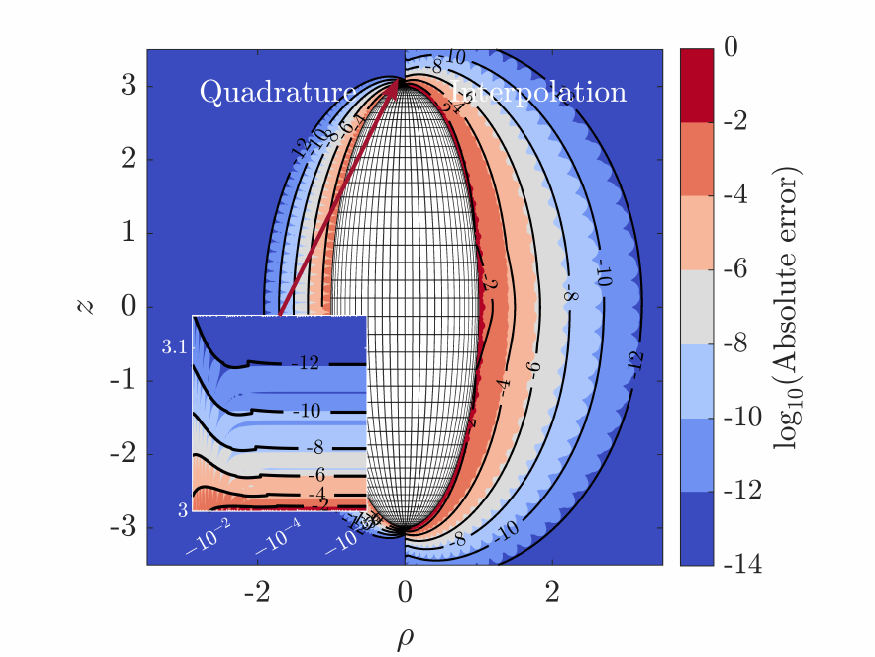}
\caption{}
\label{fig:errest_B_quad_ip_spheroid}
\end{subfigure}
\begin{subfigure}[t]{0.45\textwidth}
\includegraphics[trim={0.9cm 0.0cm 1.1cm 0.6cm},clip,width=\textwidth]{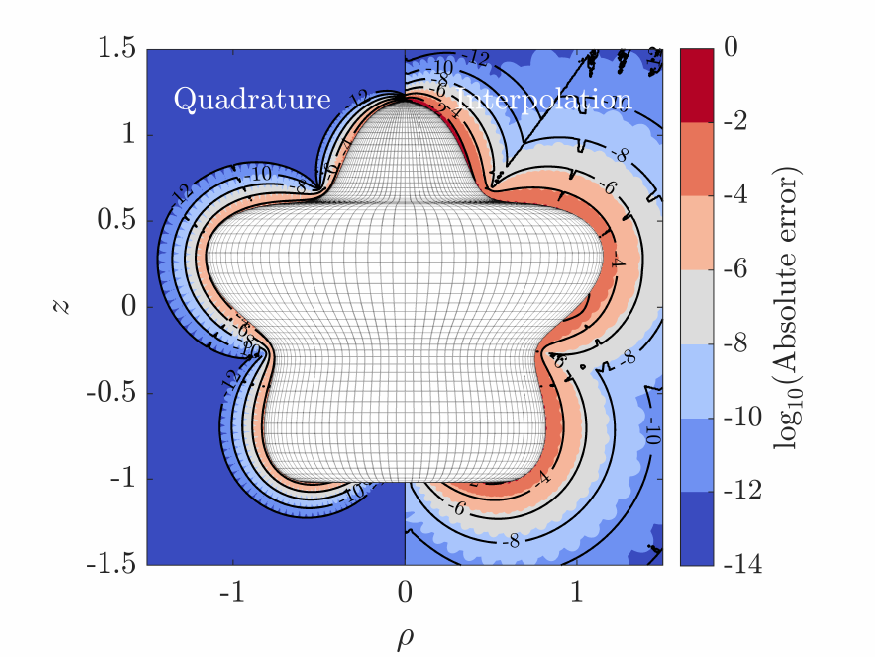}
\caption{}
\label{fig:errest_B_quad_ip_star}
\end{subfigure}
\caption{Quadrature and interpolation error prediction for $\Lsq$ in \eqref{eq:Lsq_est} on $[0,\pi]$, shown in the $\rho z$-plane. The measured errors are shown in color and the predicted levels as black contours, both on a $\log_{10}$ scale.}
\label{fig:errest_B_quad_ip}
\end{figure}

We next examine the decay of the error as $n$ increases. Since Gauss--Legendre error function \eqref{eq:GL_knz} is constant along the level sets of the Bernstein radius function $\bernsteinr$, the quadrature error is approximately constant on these level sets. We therefore construct target points by determining roots $\trootlambda$ such that $\varrho(\trootlambda)=P$, for a fixed constant $P$, and then form the corresponding target points $\xx(\trootlambda)$ following \cite[Section 5.4]{AFKLINTEBERG2022}. These target points are shown in black in Figure \ref{fig:body_targets}.

Figure \ref{fig:maxabserr_vs_nt} presents the maximum absolute error over these targets as a function of $n$. Both quadrature and interpolation errors exhibit exponential decay, with rates consistent with their theoretical dependence on the Bernstein radius. The predictors closely follow this decay, capturing both the rate and magnitude of the true error. 

To assess robustness in the close-evaluation regime, additional target points are placed along three lines normal to the surface, as illustrated in \ref{fig:body_targets}. The lower panels of Figure \ref{fig:errest_B_quad_ip}, with $n=28$ fixed, show the error as a function of target location and distance to the surface. The predictors remain stable across this range and provide reliable approximate upper bounds for both distant and near-surface target points.

\begin{figure}[!t]
\centering
\begin{subfigure}[t]{0.49\textwidth}
%[trim={left bottom right top},clip]
\includegraphics[trim={0.2cm 0.7cm 1.3cm 1.5cm},clip,width=\textwidth]{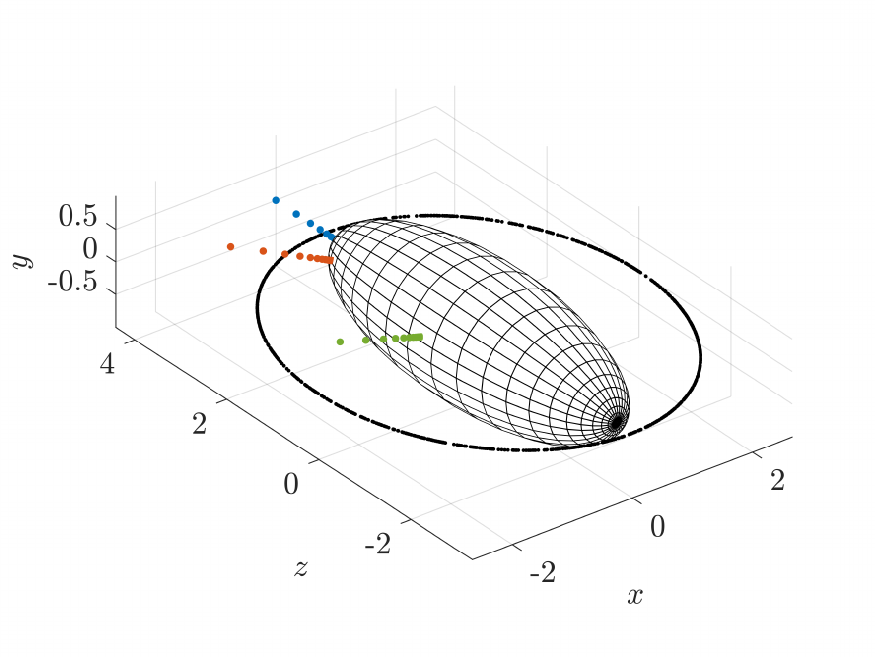}
\caption{Source geometry, 1000 black target points with $\varrho(\trootlambda)=3$ and colored target points along lines normal to the surface with fixed $\varphi=10\pi/11$ and $\theta=\{0,0.6,\pi/2\}$.}
\label{fig:body_targets}
\end{subfigure}
\hfill
\begin{subfigure}[t]{0.49\textwidth}
\includegraphics[width=\textwidth]{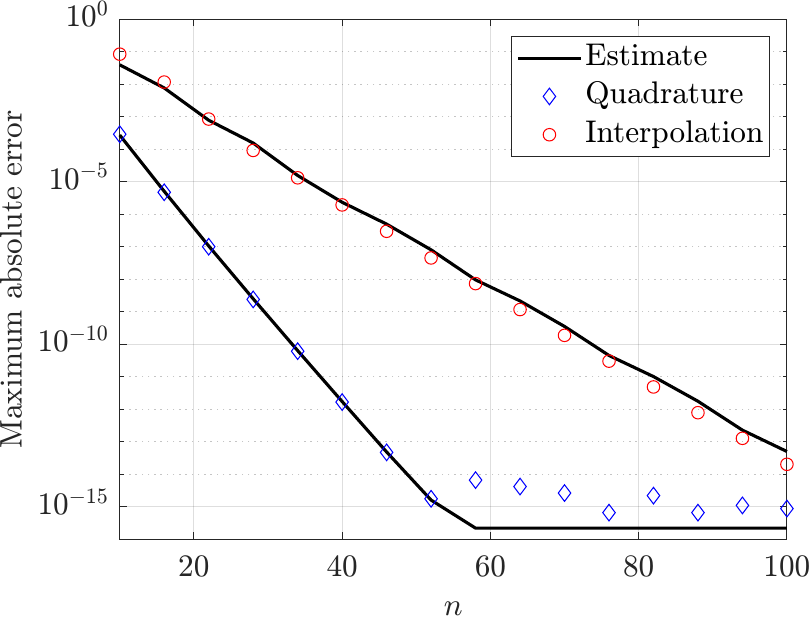}
\caption{Maximum absolute measured and predicated errors over all black target points in (a).}
\label{fig:maxabserr_vs_nt}
\end{subfigure}
\hfill
%\vskip\baselineskip
\begin{subfigure}[t]{0.49\textwidth}
\includegraphics[width=\textwidth]{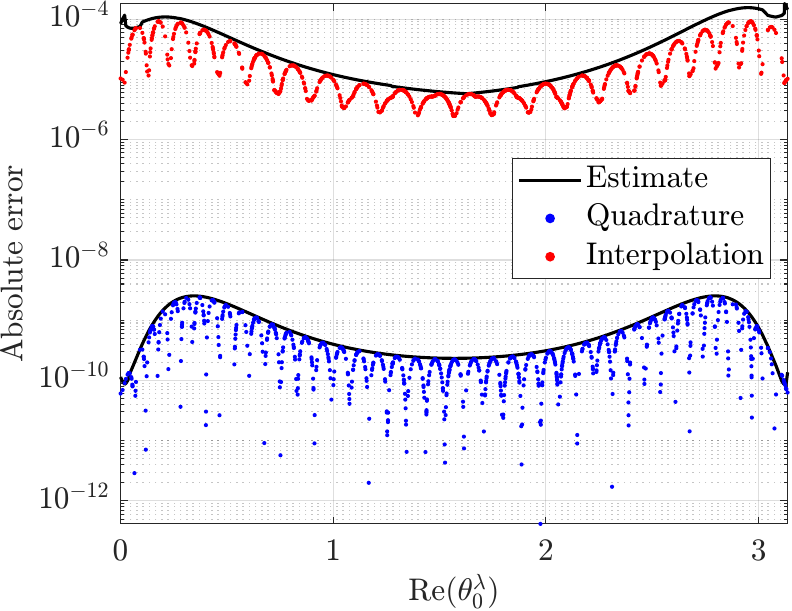}
\caption{Pointwise measured and predicted errors with $n=28$ from (b).}
\label{fig:abserr_vs_real_theta0}
\end{subfigure}
\hfill
\begin{subfigure}[t]{0.49\textwidth}
%[trim={left bottom right top},clip]
\includegraphics[width=\textwidth]{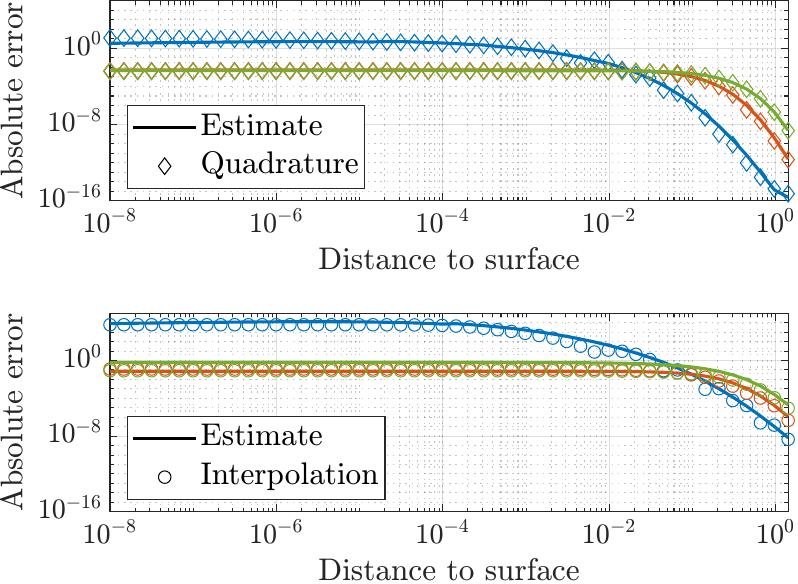}
\caption{Pointwise measured and predicted errors using $n=28$ for the colored target points in (a).}
\label{fig:Lsq_estimates_close}
\end{subfigure}
\caption{Measured and predicted quadrature and interpolation error for the function $\Lsq$ in \eqref{eq:Lsq_est} on $[0,\pi]$.}
\label{fig:B_quad_ip_conv}
\end{figure}

\section{Root finding}
\label{s:root_finding}
The error predictors in Section \ref{s:error_estimates} depend on the location of the nearest complex singularities of the relevant integrands. In our setting, these singularities are characterized by complex roots of squared-distance functions.

For a target point $\xx$, three types of roots arise:
\begin{enumerate}
\item The azimuthal root $\varphi_0(\theta,\xx)$ of the full squared-distance function $R(\theta,\varphi,\xx)^2$ for fixed $\theta$. This root is used in singularity swap quadrature (SSQ) in the azimuthal variable and is also required for evaluating the quadrature error predictor of the layer potential \eqref{eq:parameterized_layer_potential}. For axisymmetric surfaces, $\varphi_0$ is available in closed form and is given in Lemma \ref{lem:phi0}.
\item The polar root $\theta_0(\varphi,\xx)$ of $R(\theta,\varphi,\xx)^2$ for fixed $\varphi$. This root is required to evaluate the quadrature error predictor of the layer potential \eqref{eq:parameterized_layer_potential}, which is used to determine whether special quadrature is needed to achieve the desired accuracy.
\item The polar root $\theta_0^\lambda(\xx)$ of the reduced squared-distance function $R_\lambda(\theta,\xx)^2$ in \eqref{eq:Rlambda}. This root governs the error predictors for $\Lsq$ used to drive adaptivity in the polar direction and is also used in SSQ in the polar variable.
\end{enumerate}

In each case, the relevant root is the one closest to the corresponding integration interval. For the azimuthal root this interval is $[0,2\pi)$, whereas for the polar roots it is either the full interval $[0,\pi]$ or, after adaptive refinement, the current polar subpanel mapped to $[-1,1]$.
This closest-root principle reflects the fact that the contribution of a singularity to the quadrature and interpolation error decays exponentially with its Bernstein radius, which is the mechanism underlying the predictors in Section \ref{s:error_estimates}.

For spheroidal geometries, $\theta_0^\lambda(\xx)$ and $\theta_0(\varphi,\xx)$ can be obtained analytically. We begin by presenting these derivations, followed by a discussion of the general axisymmetric case.

\subsection{Analytical roots for spheroidal geometries}\label{ss:root_finding_analytic}
We consider the spheroid parametrized by
\begin{equation}
    \ggamma(\theta,\varphi) = \left(a\sin(\theta)\cos(\varphi),~a\sin(\theta)\sin(\varphi),~b\cos(\theta)\right),\quad 0\leq\theta\leq\pi,~0\leq\varphi<2\pi,
    \label{eq:gamma_spheroid}
\end{equation}
with constants $a,b>0$. The derivation of the formula in Lemma \ref{lem:theta0_spheroid} follows from reducing the problem to a planar circle and ellipse using a Joukowsky transform. For completeness, these auxiliary results are provided in Appendix \ref{a:roots}.

\begin{lemma}[Root of $\Rlambda$ for a spheroid]\label{lem:theta0_spheroid}Let $\xx=(x,y,z)\in\mathbb{R}^3$ with $\rho=\sqrt{x^2+y^2}>0$, not on the spheroid. Define
\begin{equation}
    \Rlambda(\theta,\xx) = \left(a\sin(\theta)-\rho\right)^2 + \left(b\cos(\theta)-z\right)^2.
\end{equation}
Then $\Rlambda(\theta,\xx)=0$ for $\theta=\thetarootlambda$, where
\begin{equation}
    \thetarootlambda = \atantwo(\ytilde,\xtilde) \pm i\ln\left(\frac{1}{\rho}\left(\lambda+\sqrt{\lambda^2-\rho^2}\right)\right),
    \label{eq:theta0_spheroid}
\end{equation}
with
\begin{equation}
\begin{split}
    \atilde = \frac{a+b}{2},\quad c^2 = \frac{b^2-a^2}{4}, \quad w = z+i\rho, \\
    u = \frac{1}{2}\left(w\pm\sqrt{w^2-4c^2}\right), \quad\xtilde=\Re(u), \quad\ytilde = \Im(u),
\end{split}
\end{equation}
and
\begin{equation}
    \lambda = \frac{1}{2a}\left(\atilde^2+\xtilde^2+\ytilde^2\right).
\end{equation}
Here, $\atantwo$ is defined as in Lemma \ref{lem:phi0}.
\end{lemma}

\begin{proof}
This follows directly from Lemma \ref{lem:root_ellipse} by exchanging the roles of $a$ and $b$, and letting the point be $(z,\rho)$.
\end{proof}

The following result extends the derivation in \cite{SORGENTONE2023}, where the analysis was carried out for the special case of a sphere.

\begin{lemma}[Root of $R^2$ in $\theta$ for fixed $\varphi$, spheroid]\label{lem:theta0_of_R2}Let $\xx=(x,y,z)\in\mathbb{R}^3$, not on the spheroid, and $a\neq b$. Fix $\varphi=\phibar\in[0,2\pi)$ and assume that if $z=0$ then $\phibar-\atantwo(y,x)\neq\pi/2+p\pi,~p\in\mathbb{Z}$. Define
\begin{equation}
    R^2(\theta,\phibar,\xx) = \left(a\sin(\theta)\cos(\phibar)-x\right)^2 + \left(a\sin(\theta)\sin(\phibar)-y\right)^2 + \left(b\cos(\theta)-z\right)^2.
    \label{eq:Rnew}
\end{equation}
Then $R^2(\theta,\phibar,\xx)=0$ for $\theta=\thetaroot$, where
\begin{equation} \label{eq:thetaroot_beta}
    \thetaroot = \Arg(\beta) - i \ln\left(|\beta|\right),
\end{equation}
and $\beta$ satisfies the quartic equation,
\begin{equation}
    \frac{\Delta}{4}\beta^4+\tau\beta^3+\left(\frac{\Delta}{2}+d^2\right)\beta^2+\overline{\tau}\beta+\frac{\tau}{4}= 0,
    \label{eq:quartic}
\end{equation}
with
\begin{equation}
    \Delta = b^2-a^2,\quad \tau = -bz+ia\left(x\cos(\phibar)+y\sin(\phibar)\right),\quad d = a^2+x^2+y^2+z^2,
    \label{eq:quartic_helper}
\end{equation}
and $\overline{\tau}$ denotes the complex conjugate of $\tau$.
\end{lemma}

\begin{proof}
The squared-distance function \eqref{eq:Rnew} can be written as
\begin{equation}
    R^2(\theta,\phibar,\xx) = a^2\sin^2(\theta) + b^2\cos^2(\theta) + x^2+y^2+z^2 - 2a\left(x\cos(\phibar)+y\sin(\phibar)\right)\sin(\theta) - 2bz\cos(\theta).
    \label{eq:R2lem_proof1}
\end{equation}
Next, we make the ansatz $\thetaroot=-i\eta$, $\eta=\ln(\beta)$ for some $\beta\in\mathbb{C}$ with $0=R^2(\thetaroot,\phibar,\xx)$. Inserting the ansatz into \eqref{eq:R2lem_proof1} gives
\begin{equation}
    \begin{split}
        0=\frac{1}{2}\Big(a^2+b^2+2(x^2+y^2+z^2)&-4bz\cosh(\eta)+(b^2-a^2)\cosh(2\eta) \\
        &+4ia(x\cos(\phibar)+y\sin(\phibar))\sinh(\eta)\Big).
    \end{split}
\end{equation}
Expressing the trigonometric functions in exponential form and substituting $\eta=\ln(\beta)$ yields the quartic equation \eqref{eq:quartic}-\eqref{eq:quartic_helper}. Substituting the ansatz $\thetaroot=-i\ln(\beta)$ and using the complex logarithm then gives \eqref{eq:thetaroot_beta}. Since $a\neq b$, the equation has four distinct solutions for $\beta$, producing four corresponding roots $\theta_0$. In practice, the relevant pair of complex conjugate roots are those with the smallest imaginary part in magnitude.
\end{proof}

\subsection{General axisymmetric geometries}\label{ss:root_finding_general}
The closed-form formulas above apply only to spheroids. For a general axisymmetric surface, the azimuthal root $\varphi_0(\theta,\xx)$ remains available analytically (Lemma \ref{lem:phi0}), while the polar roots $\theta_0^\lambda(\xx)$ and $\theta_0(\varphi,\xx)$ are computed numerically.

For instance, to the $\theta_0(\varphi,\xx)$ of $R(\theta,\varphi,\xx)^2=0$ at fixed $\varphi$, we use a complex Newton iteration, following the approach outlined in \cite{SORGENTONE2023}. The initial guess is chosen from the polar angle $\theta^*$ corresponding to the discretization point on the surface closest to the target $\xx$. We then set the initial guess as $\theta^*+i\beta$. With $\beta=0.1$, convergence to a strict tolerance is usually achieved within 10 iterations. 

The same procedure is used to compute $\theta_0^\lambda(\xx)$, replacing $R(\theta,\varphi,\xx)^2$ by $R_\lambda(\theta,\xx)^2$. In the rare cases where convergence is not achieved, we restart the iteration and use over-relaxation.

\section{Cost model for close evaluation}
\label{s:computational_complexity}
Here we briefly discuss the per-target cost of the presented algorithm. The goal is not to provide a sharp complexity theorem with explicit theoretical guarantees, but rather to identify the dominant operations in the current implementation and how they scale with the surface discretization and the adaptive refinement in the polar direction. 
The current MATLAB implementation is not well suited for a detailed throughput study, which is left for future work.

We assume a base surface discretization consisting of an $\nphi$-point trapezoidal rule in the azimuthal direction and an $\nt$-point Gauss--Legendre panel in the polar direction, with the layer density known at these nodes. If several base Gauss--Legendre panels are used in the polar direction, the discussion below applies panelwise. The cost of constructing the base discretization, solving for the layer density, and performing one-time precomputations is not included.

After adaptive refinement in the polar direction, the target-dependent refined grid contains $\npan\nGL$ polar nodes for each of the $\nphi$ azimuthal values, giving a total of $\nphi\npan\nGL$ refined surface points. In the current implementation, a substantial part of the cost comes from interpolating the layer density from the base polar nodes to this refined grid. For each fixed azimuthal node, this is a one-dimensional interpolation from $\nt$ source nodes to $\npan\nGL$ target nodes, resulting in a total interpolation cost of
\begin{equation}
\mathcal{O}(\nphi\npan\nGL\nt).
\end{equation}

Let $\npanSQphi$ denote the number of refined panels on which the singularity swap quadrature (SSQ) method is used in the azimuthal direction, and let $\npanTZ$ denote the number of panels on which the trapezoidal rule is sufficient, so that $\npan=\npanSQphi+\npanTZ$. For each of the $\npanSQphi\nGL$ polar nodes where azimuthal SSQ is applied, a one-dimensional $\nphi$-point FFT is computed, giving a total cost of
\begin{equation}
\mathcal{O}\left(\nphi\log(\nphi)\npanSQphi\nGL\right).
\end{equation}
After the FFTs have been computed on the panels requiring azimuthal SSQ, the azimuthal integral must still be evaluated at every refined polar node. This costs $\mathcal{O}(\nphi)$ per node, both for SSQ weights and for the trapezoidal rule, and therefore contributes
\begin{equation}
\mathcal{O}(\nphi\npan\nGL).
\end{equation}
The remaining operations in the polar direction entail evaluation of one-dimensional integrals, resulting in computational cost that is subdominant relative to these steps.

Thus, for a fixed target point $\xx$, the dominant per-target cost in the current implementation is described by
\begin{equation}
\mathcal{O}(\nphi\npan\nGL\nt) + \mathcal{O}\left(\nphi\log(\nphi)\npanSQphi\nGL\right) + \mathcal{O}\left(\nphi\npan\nGL\right).
\label{eq:full_cost}
\end{equation}
The number of $\nGL$-point Gauss–Legendre panels, $\npan$, determined by the adaptive algorithm for a given value of $\nGL$, depends primarily on the target location $\xx$ and the prescribed error tolerance $\eps$. The last term in \eqref{eq:full_cost} has the same panel-dependent factor as the interpolation term, but lacks the additional factor $\nt$, and is therefore absorbed into the first term. Accordingly, the dominant-cost model can be summarized as
\begin{equation}
    C_1(\eps,\xx,\nGL)\nt\nphi + C_2(\eps,\xx,\nGL)\nphi\log(\nphi),
    \label{eq:simplified_cost}
\end{equation}
where $C_1,~C_2>0$ capture the target-dependent refinement and the distribution of panels requiring SSQ in the azimuthal direction.

\begin{remark}
The discussion above assumes that the kernel factor $k(\xx,\yy)$ in \eqref{eq:generic_layer_potential} depends on the target point $\xx$. However, if $k(\xx,\yy)$ is independent of $\xx$, the Fourier coefficients of the corresponding smooth factor can instead be computed once in a precomputation step. Interpolation can then be performed directly on these Fourier coefficients for each of the $\npanSQphi$. In this case, the FFT cost becomes target-independent, effectively eliminating the second term in \eqref{eq:simplified_cost}.
\end{remark}

\section{Numerical experiments}
\label{s:numerical_experiments}
We now numerically demonstrate the accuracy, adaptivity, and error-control behavior of the proposed singularity swap surface quadrature (S3Q) method in Section \ref{s:s3q}. We consider several layer potentials over different smooth axisymmetric geometries and seek to evaluate them to a prescribed error tolerance $\epsilon$. Both single-panel and multi-panel base discretizations in the polar direction are used. Unless stated otherwise, a single-panel discretization is employed. In the multi-panel case, the geometry is partitioned into segments in the polar direction, which we refer to as \textit{patches}.

All code for this paper is written and run in MATLAB. This implementation is intended to demonstrate accuracy and adaptive behavior rather than optimized runtime performance.

\subsection{Experimental setup}
We consider layer potentials arising from Laplace, Helmholtz, and Stokes problems, all of which fit into the generic form \eqref{eq:generic_layer_potential} for appropriate choices of the kernel factor $k(\xx,\yy)$ and singularity power $p$. 
In all numerical experiments, the surface parametrization is assumed to be analytically known. The discretization parameters $\nt$ and $\nphi$ are chosen so that the geometry and sampled surface values are well resolved on the underlying tensor-product grid, which we refer to as the \textit{base discretization}. The layer density $\sigma$ is prescribed analytically so that it can be sampled on this grid, evaluated at complex roots for the predictors, and used to compute reference solutions. Within S3Q, the sampled density values on the base discretization are then interpolated onto the adaptively refined polar grid when needed.

Given a target point $\xx$ with the minimum distance $d$ from the surface $S$, we use the quadrature error predictor of \cite{SORGENTONE2023} to determine whether the base discretization is sufficient to evaluate $u(\xx)$ to the prescribed tolerance $\epsilon$. If the predicted error exceeds $\epsilon$, S3Q is activated and new quadrature weights are computed through the adjoint formulation described in Section \ref{ss:adjoint_method}. 
In particular, S3Q combines standard quadrature and singularity swap quadrature (SSQ) in the azimuthal and polar directions, while also allowing adaptive refinement in the polar direction using $\nGL$-point Gauss--Legendre panels. The number of such panels is denoted by $\npan$.

The error in evaluating $u(\xx)$ is measured against a reference solution. For scalar-valued layer potentials, we report the absolute error between the numerical and reference solution. For Stokes potentials, which are vector-valued, we report the maximum componentwise absolute error between the numerical and reference solution. In both settings, we refer to this as the \textit{absolute error}. The reference solution is computed using an adaptive quadrature routine.

For the Laplace experiments, the target points are sampled on a uniform $M\times M$ grid in the $xz$-plane with $M=200$ and fixed $y=0$, with points inside the geometry discarded. In the Helmholtz experiment, we consider two similar target planes, with $M=120$. In the Stokes experiment, we use the same type of target grid as in the Laplace experiments, with the only difference being that we consider targets both inside and outside the geometry.

In Section \ref{ss:helmholtz}, we also show that S3Q can be applied patchwise and coupled to a fast summation method, namely the fast multipole method (FMM) \cite{GREENGARD1987325,Greengard_Rokhlin_1997,greengard1998,CHENG1999468,GREENGARD2002642}, to accelerate the standard quadrature. Specifically, we use the point-based FMM from the FMM3D package.\footnote{The FMM3D package can be found at \href{https://github.com/flatironinstitute/FMM3D}{https://github.com/flatironinstitute/FMM3D}.}

\subsection{Harmonic single layer potential: Type S3}\label{ss:laplace_slp}
As a first example, we consider the harmonic single layer potential, corresponding to $p=1/2$,
\begin{equation}
    u(\xx) = \int_S \frac{\sigma(\yy)}{\|\yy-\xx\|}\dS(\yy)
    \label{eq:laplace_single_layer}
\end{equation}
on a type S3 spheroid. Here a type $SX$ spheroid denotes a spheroid with aspect ratio $1:X$, with semi-axes $a(\theta)=1$ and $b(\theta)=X$. The layer density function is chosen as
\begin{equation}
    \sigma(\ggamma(\theta,\varphi))=\sin(5\theta)e^{-\cos^2(\varphi)}+1.03,
    \label{eq:sigma_spheroid_s3}
\end{equation}
and the potential is evaluated for the three error tolerances $\eps=10^{-4},~10^{-6},~10^{-8}$.

Figure \ref{fig:laplace_slp_s3} presents the results with $\nt=\nphi=40$ and $\nGL=32$. The figure shows that the standard quadrature error predictor accurately identifies the target points where S3Q must be used, and that the resulting measured errors at those points are close to or below the prescribed tolerance $\eps$. The maximum absolute error over all target points is approximately $1.2\times 10^{-4}$, $5.0\times 10^{-8}$, and $5.3\times 10^{-9}$ for the tolerances $\eps=10^{-4},~10^{-6},~10^{-8}$, respectively.

As $\eps$ decreases, the number of Gauss--Legendre panels $\npan$ selected by the adaptive algorithm in the polar direction increases, as shown in Figure \ref{fig:laplace_slp_s3_histogram_nGL32}. Despite the close proximity of target points to the surface, the maximum $\npan$ required to achieve the desired error tolerances $\eps=10^{-4},~10^{-6},~10^{-8}$ was $12$, $13$ and $14$, respectively. For $\eps=10^{-8}$, this maximum was reached at only two target points situated approximately $5.2 \cdot 10^{-5}$ from the surface.

\begin{figure}[!t]
\centering
\begin{minipage}{0.42\linewidth}
\begin{subfigure}[t]{\linewidth}
%[trim={left bottom right top},clip]
\includegraphics[trim={2.1cm 0.1cm 2.0cm 0.6cm},clip,width=\linewidth]{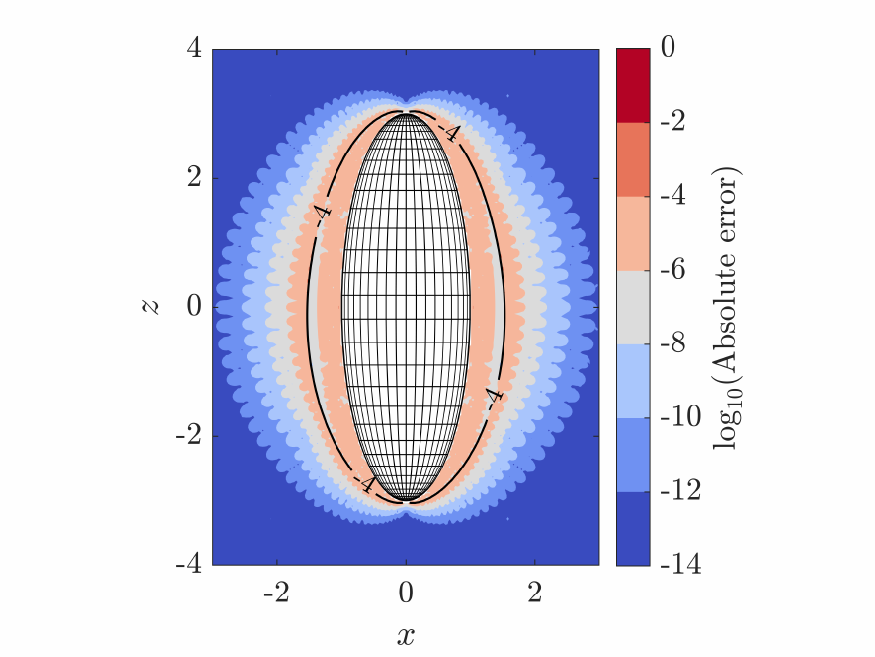}
\caption{}
\label{fig:laplace_slp_s3_error_contour}
\end{subfigure}
\end{minipage}
\begin{minipage}{0.57\linewidth}
\begin{subfigure}[t]{\linewidth}
\includegraphics[width=0.97\linewidth]{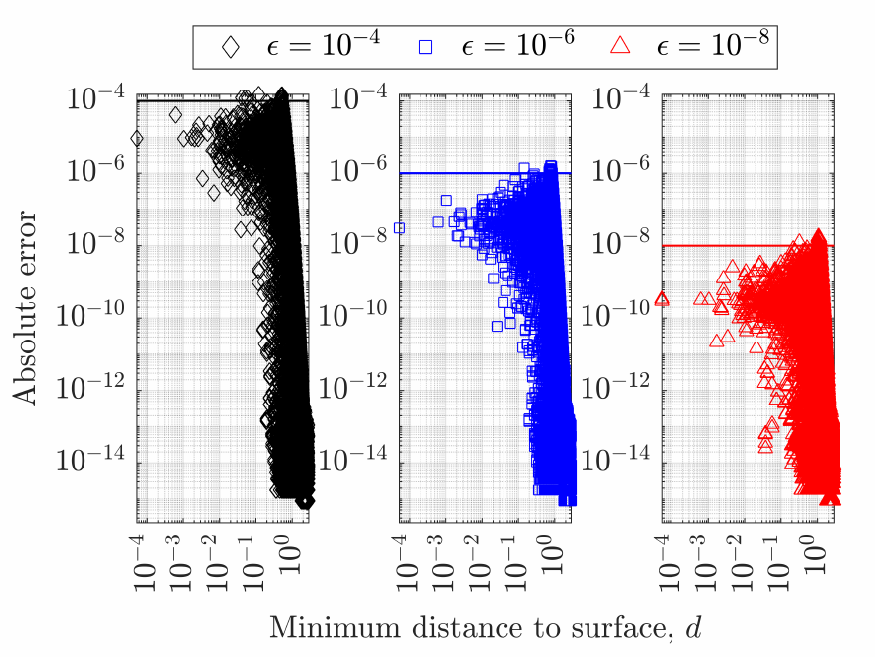}
\caption{}
\label{fig:laplace_slp_s3_err_vs_dist}
\end{subfigure}
\end{minipage}
\begin{subfigure}[t]{\linewidth}
\centering
\includegraphics[width=0.45\linewidth]{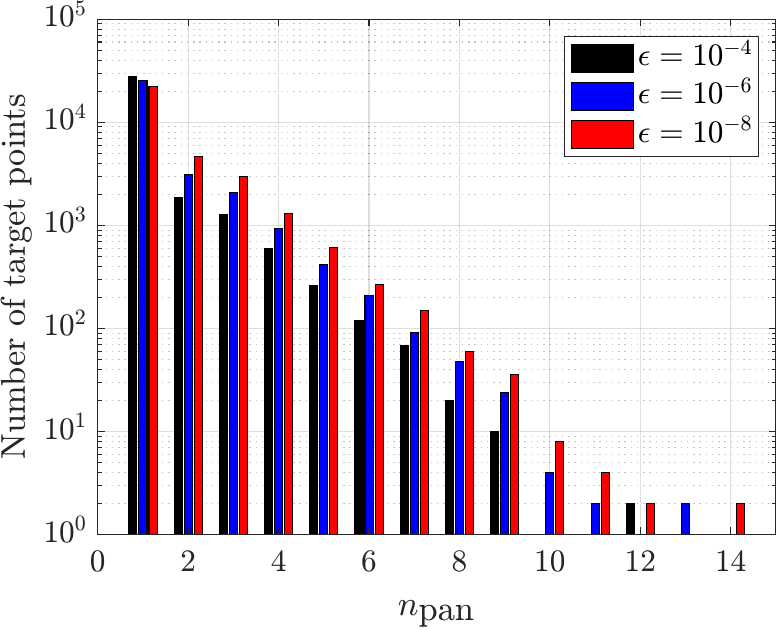}
\caption{}
\label{fig:laplace_slp_s3_histogram_nGL32}
\end{subfigure}
\caption{Harmonic single layer potential on a type S3 spheroid computed using S3Q for three prescribed error tolerances $\eps$. Panel (a) shows the measured absolute error for all target points together with the contour where the error predictor for the standard quadrature equals $\eps=10^{-4}$. Panel (b) shows the error versus the distance to the surface for all target points for each $\eps$. Panel (c) shows the distribution of the number of $\nGL$-point Gauss--Legendre panels set by the adaptive algorithm in the polar direction for each $\eps$. The bars are ordered from left to right with decreasing tolerance.}
\label{fig:laplace_slp_s3}
\end{figure}

\subsection{Harmonic double layer potential: Type S10}\label{ss:laplace_dlp}
We next demonstrate the performance of S3Q for a stronger singularity and illustrate the adaptive refinement behavior for a highly elongated geometry. Specifically, we consider the harmonic double layer potential, corresponding to $p=3/2$,
\begin{equation}
    u(\xx) = \int_S \frac{\nn_\yy\cdot(\yy-\xx)~\sigma(\yy)}{\|\yy-\xx\|^3}\dS(\yy)
    \label{eq:laplace_double_layer}
\end{equation}
evaluated on a type S10 spheroid. Here, $\nn_\yy$ denotes the outward-pointing normal vector at $\yy\in S$. The prescribed tolerances are $\eps=10^{-4},~10^{-6},~10^{-8}$, and the layer density is chosen as
\begin{equation}
\sigma(\ggamma(\theta,\varphi))=1+\sin(6\varphi+\theta)\sin(\theta)^2.
    \label{eq:sigma_spheroid_s10}
\end{equation}
The base discretization uses $n_t=160$, $n_\varphi=100$, and we set $\nGL=16$.

Figure \ref{fig:laplace_dlp_s10} shows the result for error tolerance $\eps=10^{-8}$. The measured absolute error is generally kept below the tolerance, with only 44 out of 30354 target points slightly exceeding it, and with a maximum error $2.4\times 10^{-8}$ over all target points. For $\eps=10^{-4}$ and $\eps=10^{-6}$, the corresponding maximum absolute errors are $8.3\times 10^{-5}$ and $1.6\times 10^{-6}$, respectively. The figure also shows how $\npan$ varies with the location of the target point relative to the surface.
 
This behavior is examined further in Figure \ref{fig:laplace_dlp_s10_dist_vs_npan}, where the total number of quadrature points in the polar direction, $\npan\times\nGL$, is plotted against the minimum distance $d$ from the target point to the surface for the three prescribed tolerances. The required number of points grows approximately log-linearly with $d$, and increases with stricter tolerances.

Note that the surface mesh shown in Figure \ref{fig:laplace_dlp_s10} represents the base discretization with $\nt=160$ and $\nphi=100$, whereas each adaptively introduced Gauss--Legendre panel in the polar direction contains only $\nGL=16$ points. Thus, $\npan=10$ corresponds to the same number of quadrature points in the polar direction as the base discretization.

\begin{figure}[!t]
\centering
\begin{tikzpicture}

\begin{axis}[%
width=0.6*4.521in,
height=3in,
at={(0.758in,0.405in)},
scale only axis,
point meta min=1,
point meta max=25,
xmin=0,
xmax=1,
ymin=0,
ymax=1,
axis line style={draw=none},
ticks=none,
axis x line*=bottom,
axis y line*=left,
colormap={mymap}{[1pt] rgb(0pt)=(0.19245,0,0); rgb(1pt)=(0.321029,0.170251,0.170251); rgb(2pt)=(0.411196,0.240772,0.240772); rgb(3pt)=(0.484876,0.294884,0.294884); rgb(4pt)=(0.548751,0.340503,0.340503); rgb(5pt)=(0.605929,0.380693,0.380693); rgb(6pt)=(0.658158,0.417029,0.417029); rgb(7pt)=(0.706537,0.450443,0.450443); rgb(8pt)=(0.751809,0.481543,0.481543); rgb(9pt)=(0.770846,0.545808,0.510754); rgb(10pt)=(0.789423,0.603265,0.538382); rgb(11pt)=(0.807573,0.655707,0.56466); rgb(12pt)=(0.825324,0.704254,0.589768); rgb(13pt)=(0.842701,0.749664,0.61385); rgb(14pt)=(0.859727,0.792477,0.637022); rgb(15pt)=(0.876422,0.833092,0.65938); rgb(16pt)=(0.892805,0.871817,0.681005); rgb(17pt)=(0.908893,0.908893,0.701964); rgb(18pt)=(0.924701,0.924701,0.759799); rgb(19pt)=(0.940244,0.940244,0.813533); rgb(20pt)=(0.955533,0.955533,0.863931); rgb(21pt)=(0.970582,0.970582,0.911547); rgb(22pt)=(0.985401,0.985401,0.956796); rgb(23pt)=(1,1,1)},
colorbar horizontal,
colorbar style={at={(0.5,1.03)}, anchor=south, xticklabel pos=upper, xlabel style={font=\color{white!15!black}}, xlabel={$n_{\textrm{pan}}$}},
colorbar sampled,
colorbar style={samples=25}
]
\end{axis}
\end{tikzpicture}%
\begin{tikzpicture}

\begin{axis}[%
width=0.6*4.521in,
height=3in,
at={(0.758in,0.405in)},
scale only axis,
point meta min=-14,
point meta max=0,
xmin=0,
xmax=1,
ymin=0,
ymax=1,
axis line style={draw=none},
ticks=none,
axis x line*=bottom,
axis y line*=left,
colormap={mymap}{[1pt] rgb(0pt)=(0.230033,0.298999,0.754002); rgb(1pt)=(0.438654,0.574065,0.953911); rgb(2pt)=(0.667541,0.779704,0.993653); rgb(3pt)=(0.865395,0.86541,0.865396); rgb(4pt)=(0.969053,0.721418,0.612449); rgb(5pt)=(0.908216,0.459395,0.358028); rgb(6pt)=(0.705998,0.0161265,0.150001)},
colorbar horizontal,
colorbar style={at={(0.5,1.03)}, anchor=south, xticklabel pos=upper, xlabel style={font=\color{white!15!black}}, xlabel={$\log_{10}(\textrm{Absolute error})$}},
colorbar sampled,
colorbar style={samples=8}
]
\end{axis}
\end{tikzpicture}%
\hspace{-1.1cm}
%[trim={left bottom right top},clip]
\begin{tikzpicture}[remember picture,overlay]
\node [xshift=0cm,yshift=7.1cm] at (current page.center){\includegraphics[trim={5.4cm 0.15cm 11.1cm 1.1cm},clip,width=\textwidth]{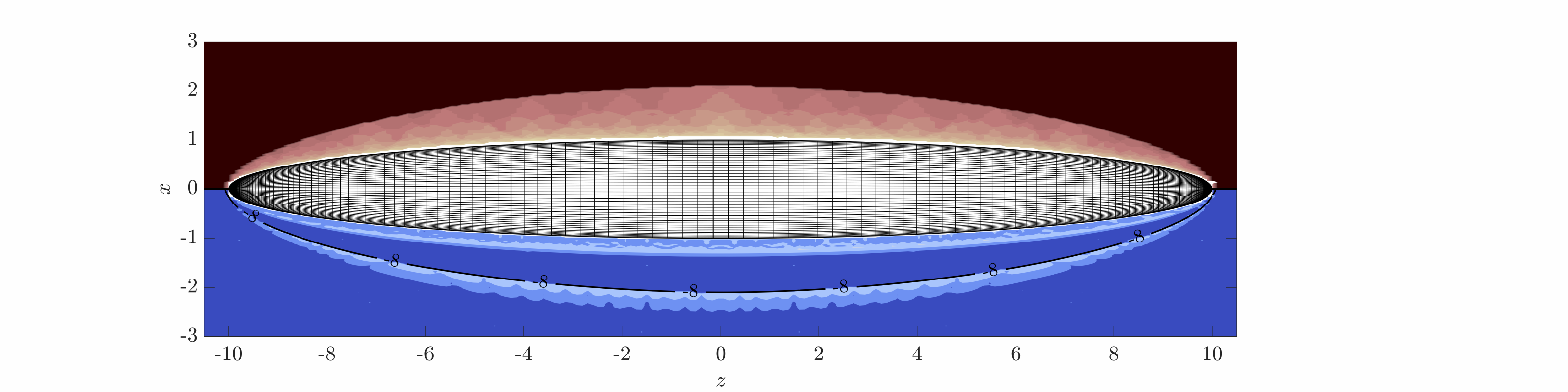}};
\end{tikzpicture}
\vspace{-2.3cm}\caption{Harmonic double layer potential on a type S10 spheroid computed using S3Q for error tolerance $10^{-8}$. The upper part shows the number of $\nGL$-point Gauss--Legendre panels, $\npan$, set by the adaptive algorithm in the polar direction, while the lower part shows the measured absolute error.}
\label{fig:laplace_dlp_s10}
\end{figure}

\begin{figure}[!t]
\centering
\includegraphics[width=0.45\textwidth]{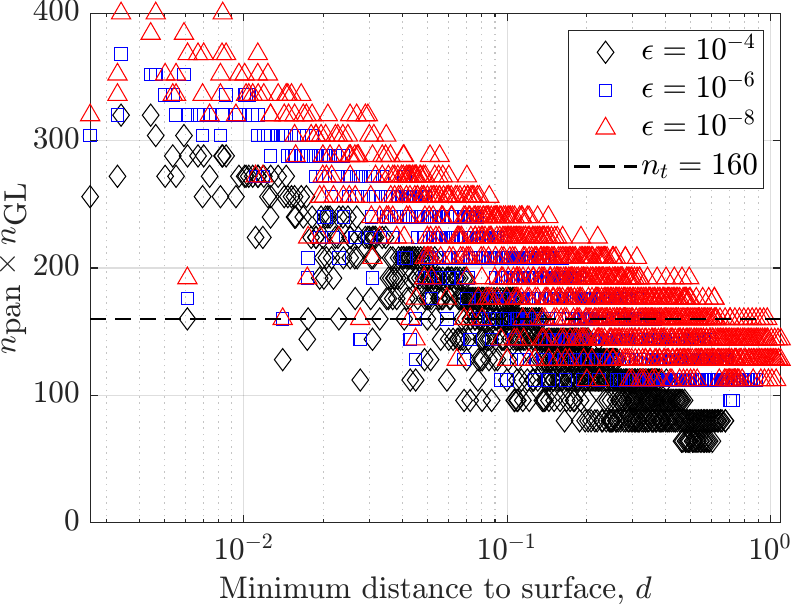}
\caption{Total number of discretization points in the polar direction set by the adaptive algorithm for prescribed error tolerance $\eps$. The dashed line indicates the number of points in the base discretization in the polar direction. The problem setup is as in Figure \ref{fig:laplace_dlp_s10}.}
\label{fig:laplace_dlp_s10_dist_vs_npan}
\end{figure}

\subsection{Helmholtz double layer potential: Multiple type S3}\label{ss:helmholtz}
We next consider a Helmholtz double layer potential and illustrate how S3Q can be applied patchwise and coupled to fast summation by using the near-quadrature correction together with an FMM-accelerated standard quadrature. We consider the Helmholtz double layer potential on a collection of type S3 spheroids, denoted by $\mathcal{S}$, given by
\begin{equation}
u(\xx) = \frac{1}{4\pi}\int_{\mathcal{S}} \frac{\left(1-i\omega\|\yy-\xx\|\right)e^{i\omega\|\yy-\xx\|}~\nn_{\yy}\cdot(\xx-\yy)~\sigma(\yy)}{\|\yy-\xx\|^3}\dS(\yy),
\end{equation}
where $\omega$ denotes the wavenumber and $\nn_{\yy}$ the outward-pointing unit normal at $\yy\in\mathcal{S}$.

To apply S3Q, we split the Helmholtz kernel into a non-smooth part and a smooth part, following \cite{HELSING2014686}. More precisely, we write the Helmholtz double layer kernel as
\begin{equation}
    \mathcal{D}^H(\xx,\yy) = D(\xx,\yy)\left(H_1(\xx,\yy)+iH_2(\xx,\yy)\right),
\end{equation}
where
\begin{equation}
    D(\xx,\yy) = \frac{\nn_\yy\cdot(\xx-\yy)}{4\pi\|\yy-\xx\|^3}
\end{equation}
and
\begin{equation}
    H_1(\xx,\yy) = \cos(\omega\|\yy-\xx\|) + \omega\|\yy-\xx\|\sin(\omega\|\yy-\xx\|),
\end{equation}
\begin{equation}
    H_2(\xx,\yy) = \sin(\omega\|\yy-\xx\|) - \omega\|\yy-\xx\|\cos(\omega\|\yy-\xx\|).
\end{equation}
The factor $DH_1$ contains the non-smooth part of the kernel, while $DH_2$ is smooth. Indeed, expanding $H_1$ about $\|\yy-\xx\|=0$ shows that $DH_1$ consists of one contribution of type $p=1/2$ and one contribution of type $p=3/2$, and therefore falls within the class of nearly singular kernels treated by S3Q. Thus, S3Q is applied only to these two non-smooth contributions of $DH_1$, whereas the smooth contribution $DH_2$ is evaluated by the standard quadrature, accelerated by the FMM.\footnote{The Helmholtz single layer potential admits an analogous splitting; see \cite{HELSING2014686}. The same idea also applies to Yukawa single and double layer kernels, which similarly enables the use of S3Q.}

The geometry consists of the cluster of 40 type S3 spheroids shown in Figure \ref{fig:helmholtz_s3q_err}. The base discretization for each spheroid has three panels in the polar direction, each with $\nt=10$, and $\nphi=26$ points in the azimuthal direction. Figure \ref{fig:helmholtz_spheroid} shows one such spheroid. In this experiment, the wavenumber is set to $\omega=4.2$, the layer density is chosen as
\begin{equation}
    \sigma(\yy) = 2+\sin(y_3),\quad y=\left(y_1,y_2,y_3\right),
\end{equation}
and the prescribed tolerance is $\eps=10^{-6}$. The full Helmholtz double layer potential is first evaluated using the FMM, after which S3Q is used as a local near-quadrature correction for the two non-smooth terms in $DH_1$. Both the S3Q corrections and the quadrature error predictors of \cite{SORGENTONE2023} are applied patchwise, so that each patch is treated independently when determining whether correction is required.

Figure \ref{fig:helmholtz_s3q_err} shows the result. Of the 28800 target points in the grid, 27193 lie outside the particles, and 5257 of these exterior targets, or 19.3\%, are classified as requiring S3Q. The maximum absolute error over all exterior target points is $3.0\times 10^{-6}$. 
Excluding the cost of evaluating the quadrature error predictors, approximately 28\% of the runtime is spent in the FMM and 72\% in constructing the S3Q near-quadrature weights.

\begin{figure}[!t]
\centering
\begin{minipage}{0.34\linewidth}
\begin{subfigure}[t]{\linewidth}
%[trim={left bottom right top},clip]
\includegraphics[trim={4.0cm 2.0cm 4.65cm 1.6cm},clip,width=0.85\linewidth]{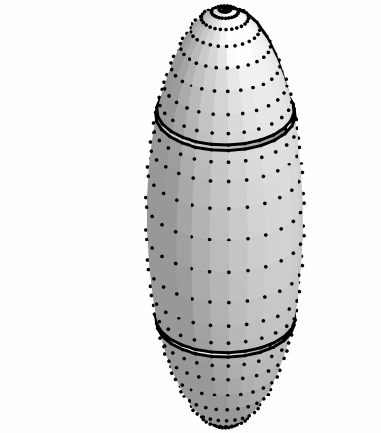}
\caption{}
\label{fig:helmholtz_spheroid}
\end{subfigure}
\end{minipage}
\begin{minipage}{0.65\linewidth}
\begin{subfigure}[t]{\linewidth}
\includegraphics[width=1\linewidth]{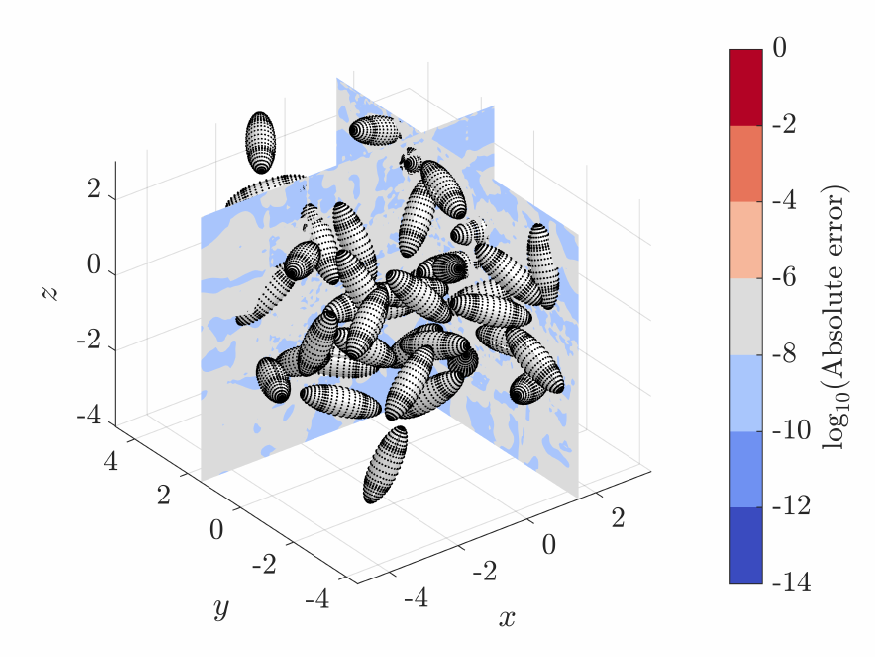}
\caption{}
\label{fig:helmholtz_s3q_err}
\end{subfigure}
\end{minipage}
\caption{Panel (a) shows a type S3 spheroid with a three-patch base discretization in the polar direction. Panel (b) shows the measured absolute error for the Helmholtz double layer potential on a collection of 40 type S3 spheroids. The potential is evaluated using FMM together with S3Q, with prescribed tolerance $\eps=10^{-6}$.}
\label{fig:helmholtz}
\end{figure}

\begin{figure}[!t]
\centering
\begin{subfigure}[t]{0.33\textwidth}
%[trim={left bottom right top},clip]
\includegraphics[trim={3.0cm 0.2cm 3.1cm 0.55cm},clip,width=\linewidth]{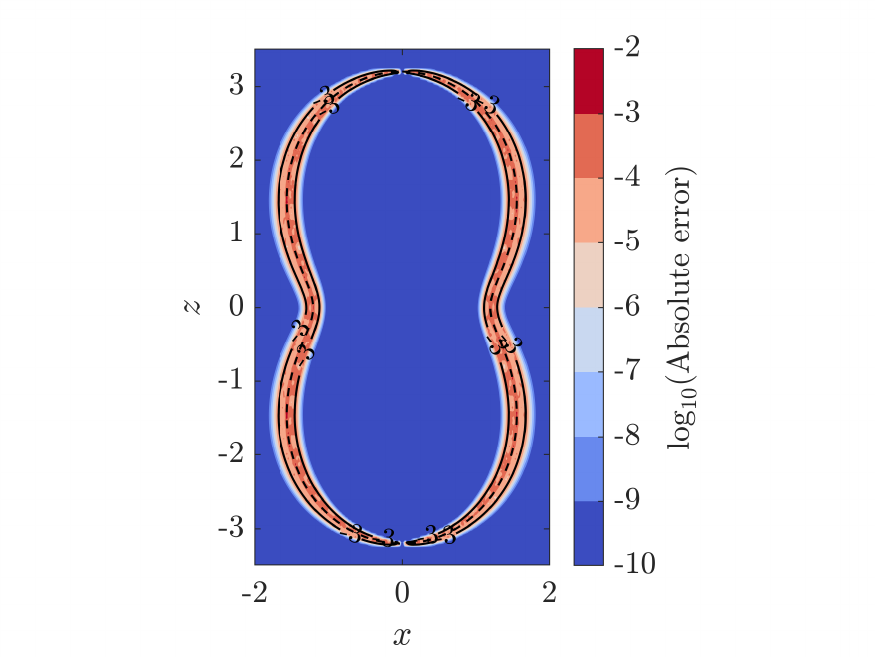}
\caption{}
\label{fig:stokes_dlp_stresslet_id_peanut_tol3}
\end{subfigure}
\begin{subfigure}[t]{0.33\textwidth}
\includegraphics[trim={3.0cm 0.2cm 3.1cm 0.55cm},clip,width=\linewidth]{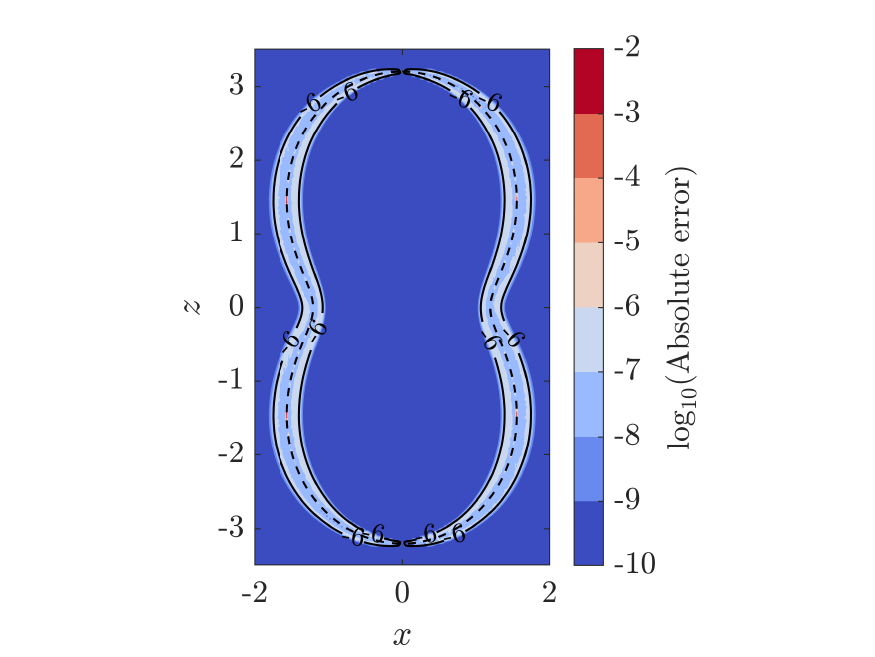}
\caption{}
\label{fig:stokes_dlp_stresslet_id_peanut_tol6}
\end{subfigure}
\begin{subfigure}[t]{\textwidth}
\centering
\includegraphics[width=0.47\linewidth]{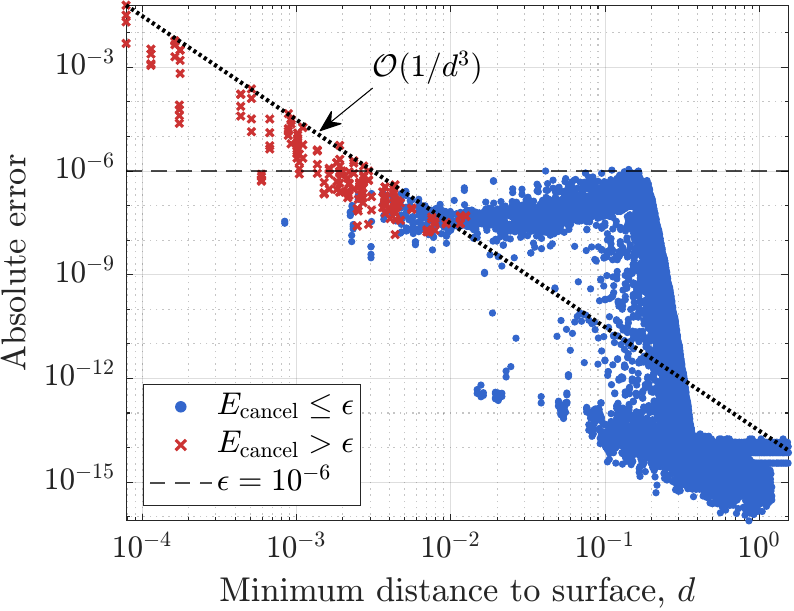}
\caption{}
\label{fig:stokes_dlp_stresslet_id_peanut_tol6_dist_err}
\end{subfigure}
\caption{Panels (a) and (b) show the measured absolute stresslet identity error for the Stokes double layer potential evaluated using S3Q with prescribed tolerances $\eps=10^{-3}$ and $10^{-6}$, respectively. The dashed black line indicates the peanut-shaped surface, and the solid black lines show the contour where the standard quadrature error predictor equals $\eps$. Panel (c) shows the measured absolute error versus the minimum distance $d$ to the surface for $\eps=10^{-6}$. Target points for which the predicted cancellation error $E_{\textrm{cancel}}$ exceeds $\eps$ are marked with crosses.}
\label{fig:stokes_dlp_stresslet_id_peanut}
\end{figure}

\subsection{Stokes double layer potential: Axisymmetric geometry}\label{ss:stokes_dlp}
We now evaluate the Stokes double layer potential $\uu$, corresponding to $p=5/2$, on an axisymmetric ``peanut'' geometry using S3Q. The three components of $\uu$ are given by\footnote{The Einstein summation convention is applied here, where summation over repeated indices in a term is implied.}, for $i=1,2,3$,
\begin{equation}
    u_i(\xx) = \int_S \mathcal{T}_{ijk}(\xx-\yy)\sigma_j(\yy)n_k(\yy)\dS,\quad \mathcal{T}_{ijk}(\rr) = -6\frac{r_ir_jr_k}{\left\|\rr\right\|^5},
    \label{eq:stokes_dlp}
\end{equation}
where $\mathcal{T}_{ijk}$ is the stresslet tensor. 
We take the constant density $\sigma_j(\yy)=1$, $j=1,2,3$, for which the stresslet identity gives an exact reference solution.
\footnote{For any constant density vector $\Tilde{\boldsymbol{\sigma}}$ the Stokes double layer potential satisfies $\uu(\xx)=\mathbf{0}$ outside $S$, $\uu(\xx)=4\pi\Tilde{\boldsymbol{\sigma}}$ on $S$, and $\uu(\xx)=8\pi\Tilde{\boldsymbol{\sigma}}$ inside $S$; see \cite{Pozrikidis1992}.} 
The prescribed tolerances are $\eps=10^{-3}$ and $\eps=10^{-6}$, and the base discretization uses $\nt=320$, $\nphi=180$, and we set $\nGL=16$. 

Figure \ref{fig:stokes_dlp_stresslet_id_peanut_tol3} and \ref{fig:stokes_dlp_stresslet_id_peanut_tol6} show the measured absolute error for $\eps=10^{-3}$ and $\eps=10^{-6}$, respectively. 
For $\eps=10^{-3}$, the measured error remains close to the requested accuracy over nearly the entire target set, with a maximum absolute error of $4.4\times 10^{-2}$, attained at the target point closest to the surface, for which $d=7.9\times 10^{-5}$. For $\eps=10^{-6}$, the same overall behavior is observed, except in a thin region nearest the surface where the prescribed accuracy is not achieved.

This loss of accuracy for very close targets is caused by \textit{catastrophic cancellation} in the azimuthal SSQ step. The issue arises when the kernel has both a strong singularity and a near-vanishing numerator, and becomes more severe as the singularity strengthens. In this regime, the true value of the integral is small, while the SSQ quadrature sum is formed from much larger terms that nearly cancel, leading to substantial amplification of roundoff errors. This behavior was already noted in the original SSQ paper \cite[Remark 11]{AFKLINTEBERG2021}, and has recently been analyzed in detail in \cite{krantz2025stabilizingsingularityswapquadrature}, where it is shown to be tied to the use of a fixed interpolation basis that does not reflect the local vanishing structure of the numerator.

To examine this more directly, let $E_{\textrm{cancel}}$ denote the approximate cancellation error introduced in \cite[Remark 2.1]{krantz2025stabilizingsingularityswapquadrature}. Figure \ref{fig:stokes_dlp_stresslet_id_peanut_tol6_dist_err} shows that the targets for which $E_{\textrm{cancel}}>\eps$ generally coincide with those whose errors exceed $\eps$. 
Thus, the observed loss of accuracy for the closest targets is well explained by floating-point cancellation in the azimuthal SSQ step, rather than by a failure of the adaptive refinement or error-control mechanisms of S3Q. This is precisely the failure mode addressed by the translated-basis stabilization developed in \cite{krantz2025stabilizingsingularityswapquadrature}.

\section{Conclusions}
\label{s:conclusions}
We have presented an adaptive approach for evaluating nearly singular layer potentials on smooth axisymmetric surfaces, in which standard quadrature is used whenever possible and the proposed \textit{singularity swap surface quadrature} (S3Q) is applied only where necessary. A key feature of the method is that all algorithmic choices and parameter values associated with the near-quadrature correction are determined automatically from the prescribed error tolerance.

The S3Q method combines singularity swap quadrature (SSQ) with a local adaptive discretization procedure. It requires only the prescribed error tolerance as essential user input, together with, optionally, the order of the Gauss--Legendre quadrature used for adaptive refinement in the polar direction. The near-quadrature correction employs SSQ in the azimuthal direction together with adaptive refinement in the polar direction, where either standard Gauss--Legendre quadrature or SSQ is used on the refined grid depending on the predicted quadrature error. After the azimuthal correction, the remaining integrand in the polar direction typically has a nearly singular behavior of one degree lower than that of the original layer potential kernel. This integral is then evaluated using SSQ only when the remaining near singularity demands it, and otherwise by standard Gauss--Legendre quadrature. A central component of the method is the use of quadrature and interpolation error predictors to control the local refinement. In particular, the interpolation error prediction extends the complex-variable framework of \cite{AFKLINTEBERG2022}, previously used for quadrature error prediction in \cite{SORGENTONE2023}, and provides a practical mechanism for automatic parameter selection.

The numerical experiments show that this strategy provides reliable error control for harmonic single and double layer potentials, while requiring only moderate refinement in the polar direction even for targets close to the surface. They further show that the framework extends naturally to Helmholtz problems through kernel splitting, and that in the corresponding multi-body example the S3Q correction can be applied patchwise while the standard quadrature is accelerated through coupling to the fast multipole method.

For the Stokes double layer potential, the experiments show that the combination of standard quadrature and S3Q remains accurate over most of the close-evaluation region, but that a very thin innermost layer becomes dominated by floating-point cancellation in the azimuthal SSQ step. In the present experiments, this limitation is primarily visible for the $1/r^5$ singularity and only for very close targets when high accuracy is requested.

A robust treatment of this regime is outside the scope of the present paper. However, the recent stabilization developed in \cite{krantz2025stabilizingsingularityswapquadrature} shows that this cancellation can be greatly reduced by replacing the standard SSQ basis with target-adapted translated bases. In the present surface setting, this suggests a natural next step: to combine the near-quadrature correction developed here with a stabilized azimuthal SSQ evaluator. This, together with a more optimized implementation of the method, is left for future work.

\section*{Acknowledgments}
The authors would like to thank Ludvig af Klinteberg at the Department of Business and Mathematics, Mälardalen University (MDU), for helpful discussions. The authors acknowledge the support from the Swedish Research Council under grant 2023-04269.

\appendix

\section{Two proofs}\label{a:proofs}\begin{proof}[Proof of Lemma \ref{lem:mu}]
Let $\varphi_0=\alpha+i\beta$ with $\beta>0$, $\chi=e^{-\beta}$. With the change of variables $\psi=\varphi-\alpha$, we have
\begin{equation}
S_k^p(\varphi_0) = e^{ik\alpha} \int_0^{2\pi} \frac{e^{ik\psi}}{(1+\chi^2-2\chi\cos\psi)^p}\D\psi.
\end{equation}
The imaginary part of the remaining integral vanishes by symmetry. Hence, for $k\geq 0$,
\begin{equation}
S_k^p(\varphi_0) = 2e^{ik\alpha}\omega_k^p(\chi), \qquad \omega_k^p(\chi) = \int_0^\pi\frac{\cos(k\psi)}{(1-2\chi\cos(\psi)+\chi^2)^p}\D\psi .
\label{eq:omega_def}
\end{equation}
We rewrite the denominator of the integrand of $\omega_k^p$ as
\begin{equation}
\chi^2-2\chi\cos(\psi)+1=(1-\chi)^2\left(1+\frac{4\chi}{(1-\chi)^2}\sin^2(\psi/2)\right) = (1-\chi)^2\left(1-\ell(\chi)\sin^2(\psi/2)\right),
\label{eq:omega1}
\end{equation}
where
\begin{equation}
\ell(\chi)=-\frac{4\chi}{(1-\chi)^2}.
\end{equation}
Substituting \eqref{eq:omega1} into $\omega_k^p$, utilizing the periodicity of the integrand and the variable substitution $t=\psi/2$ we find
\begin{equation}
\omega_k^p(\chi) = \frac{1}{(1-\chi)^{2p}}\underbrace{\int_0^\pi\frac{\cos(2kt)}{\left(1-\ell(\chi)\sin^2(t)\right)^{p}}\dt}_{\eqqcolon Q_k^p(\chi)}.
\label{eq:omega2}
\end{equation}
It remains to derive recurrences for $Q_k^p$. We temporarily omit the argument $\chi$ of $\ell$ and $Q_k^p$. For $k\geq 1$,
\begin{equation}
Q_k^p = \underbrace{\int_0^\pi\frac{(1-2\sin^2(t))\cos(2(k-1)t)}{\left(1-\ell\sin^2(t)\right)^p}\dt}_{\eqqcolon R_k^p} - \underbrace{\int_0^\pi\frac{\sin(2t)\sin(2(k-1)t)}{\left(1-\ell\sin^2(t)\right)^p}\dt}_{\eqqcolon T_k^p},
\end{equation}
where we used the identities
\begin{equation}
\cos(2kt) = \cos(2t)\cos(2(k-1)t) - \sin(2t)\sin(2(k-1)t), \qquad \cos(2t) = 1-2\sin^2(t).
\end{equation}
We treat the two integrals $R_k^p$ and $T_k^p$ separately. For the first, adding and subtracting $(2/\ell)Q_{k-1}^p$ gives
\begin{equation}
R_k^p = \frac{2}{\ell}Q_{k-1}^{p-1} + \frac{\ell-2}{\ell}Q_{k-1}^{p}.
\label{eq:Rkp}
\end{equation}
For the second term, integration by parts gives
\begin{equation}
T_k^p = -\frac{2(k-1)}{(p-1)\ell}\int_0^\pi\frac{\cos(2(k-1)t)}{\left(1-\ell\sin^2(t)\right)^{p-1}} = -\frac{2(k-1)}{\ell(p-1)}Q_{k-1}^{p-1},
\label{eq:Tkp}
\end{equation}
where the boundary term vanishes. Combining \eqref{eq:Rkp} and \eqref{eq:Tkp} yields, for $p>1/2$,
\begin{equation}
Q_k^p = \frac{1+\chi^2}{2\chi}Q_{k-1}^p - \frac{(1-\chi)^2}{2\chi}\frac{p+k-2}{p-1}Q_{k-1}^{p-1}.
\label{eq:Qkp}
\end{equation}

For $p=1/2$, the recurrence \eqref{eq:Qkp} would involve $Q_{k-1}^{-1/2}$. To eliminate this term, we now find a relation between $Q_k^{1/2}$ and $Q_k^{-1/2}$. By partial integration and a trigonometric identity it follows that
\begin{equation}
\begin{split}
Q_k^{-1/2} &= \frac{\ell}{4k}\int_0^\pi\frac{\sin(2kt)\sin(2t)}{\left(1-\ell\sin^2(t)\right)^{1/2}}\dt \\
&= \frac{\ell}{8k}\int_0^\pi\frac{\cos(2(k-1)t)}{\left(1-\ell\sin^2(t)\right)^{1/2}}\dt - \frac{\ell}{8k}\int_0^\pi\frac{\cos(2(k+1)t)}{\left(1-\ell\sin^2(t)\right)^{1/2}}\dt \\
&= \frac{\ell}{8k}\left(Q_{k-1}^{1/2}-Q_{k+1}^{1/2}\right).
\end{split}
\label{eq:Qk1}
\end{equation}
Substituting \eqref{eq:Qk1} into \eqref{eq:Qkp} with $p=1/2$ gives
\begin{equation}
Q_k^{1/2} = \frac{1+\chi^2}{\chi}\frac{2(k-1)}{2k-1}Q_{k-1}^{1/2} - \frac{2k-3}{2k-1}Q_{k-2}^{1/2},
\label{eq:Qk1new}
\end{equation}
which, as before, requires $Q_{k-1}^{1/2}$, but now also $Q_{k-2}^{1/2}$. Thus, combining \eqref{eq:Qkp} and \eqref{eq:Qk1new} yields
\begin{equation}
    Q_k^p(\chi) =
    \begin{cases}
        \dfrac{1+\chi^2}{\chi}\dfrac{2(k-1)}{2k-1}Q_{k-1}^{p} - \dfrac{2k-3}{2k-1}Q_{k-2}^{p},\quad & p=1/2, \\
        \dfrac{1+\chi^2}{2\chi}Q_{k-1}^p - \dfrac{(1-\chi)^2}{2\chi}\dfrac{p+k-2}{p-1}Q_{k-1}^{p-1},\quad & p>1/2.
    \end{cases}
    \label{eq:Qkp_rec}
\end{equation}

The initial values of $Q_k^p$ for $p=1/2,3/2,5/2$ are computed in closed form using Mathematica \cite{Mathematica}. 
In each case they contain a common factor $(1-\chi)$. We therefore define $\mu_k^p(\chi)=Q_k^p(\chi)/(1-\chi)$.
Substituting this into \eqref{eq:omega2} gives
\begin{equation}
\omega_k^p(\chi) = \frac{\mu_k^p(\chi)}{(1-\chi)^{2p-1}}.
\label{eq:omega3}
\end{equation}
The recurrence formulas for $\mu_k^p$ are the same as those for $Q_k^p$ in \eqref{eq:Qkp_rec}, and the initial values are those stated in the lemma.

Finally, it follows that $\mu_{-k}^p(\chi)=\mu_k^p(\chi)$. Moreover, since the denominator in $S_k^p(\varphi_0)$ is real-valued for real $\varphi$, $S_{-k}^p(\varphi_0)=\overline{S_k^p(\varphi_0)}$. Together with \eqref{eq:omega_def} and \eqref{eq:omega3}, this proves the stated formula for all $k\in\mathbb{Z}$.
\end{proof}

\begin{proof}[Proof of Lemma \ref{lem:nu_rec}]
Rewrite the integrals $\nu_k^p(\troot)$ as
\begin{equation}
    \nu_k^p(\troot) = \int_{-1}^1 \frac{t^{k-1}}{\left((t-t_r)^2+t_i^2\right)^{p-1/2}}\dt.
    \label{eq:appendix_nukp}
\end{equation}
Consider $p=3/2$. Then,
\begin{equation}
\begin{split}
    \nu_{k+1}^{3/2}(\troot) &= \int_{-1}^1 \frac{t^{k-2}t^2}{(t-t_r)^2+t_i^2}\dt = \int_{-1}^1 \frac{t^{k-2}\left((t-t_r)^2+t_i^2+2t_rt-(t_r^2+t_i^2)\right)}{(t-t_r)^2+t_i^2}\dt \\
    &= \int_{-1}^1t^{k-2}\dt + 2t_r\int_{-1}^1\frac{t^{k-1}}{(t-t_r)^2+t_i^2}\dt - |\troot|^2\int_{-1}^1\frac{t^{k-2}}{(t-t_r)^2+t_i^2}\dt \\
    &= \frac{1-(-1)^{k-1}}{k-1} + 2t_r\nu_k^{3/2}(\troot) - |\troot|^2\nu_{k-1}^{3/2}(\troot).
\end{split}
\end{equation}
By shifting the indices by one, we arrive in the recursion formula in \eqref{eq:nu_rec}. The recurrence formula for $p>3/2$ is found analogously. Note that if $\troot=t_r+it_i$ is purely imaginary, i.e.~$t_r=0$, and $k$ is even, the integrand in \eqref{eq:appendix_nukp} vanishes due to the odd symmetry of the integrand and the symmetric integration interval.

Lastly, the initial values for $\nu_k^p(\trootlambda)$ in \eqref{eq:nu_rec} for $p=3/2,~5/2$, computed using Mathematica \cite{Mathematica}, are
\begin{equation}
    \nu_1^{3/2}(\trootlambda) =
    \begin{cases}
        \dfrac{1}{t_i}\left(\tan^{-1}\left(\dfrac{1-t_r}{t_i}\right)+\tan^{-1}\left(\dfrac{1+t_r}{t_i}\right)\right),& \text{ if } t_i\neq 0, \\
        \dfrac{2}{t_r^2-1}, & \text{ if } t_i=0,
    \end{cases}
    \label{eq:nu_rec_initial_value_1}
\end{equation}
\begin{equation}
    \nu_2^{3/2}(\trootlambda)=
    \begin{cases}
        \dfrac{1}{2}\log\left(1-\dfrac{4t_r}{(1+t_r)^2+t_i^2}\right) + t_r\nu_0^{3/2}(\trootlambda),& \text{ if } t_i\neq 0, \\
        t_r\nu_1^{3/2}(\trootlambda) + \log\left(\dfrac{t_r-1}{t_r+1}\right),& \text{ if } t_i=0,
    \end{cases}
    \label{eq:nu_rec_initial_value_2}
\end{equation}
\begin{equation}
    \nu_1^{5/2}(\trootlambda) = 
    \begin{cases}
    \begin{aligned}
    &\frac{1}{2t_i^3\left((-1+t_r)^2+t_i^2\right)\left((1+t_r)^2+t_i^2\right)}\left(2t_i(1-t_r^2+t_i^2)\right. \\
    &\left. +(1+t_r^4+t_i^4)\tan^{-1}\left(\frac{1-t_r}{t_i}\right) - 2\left(t_i^2+t_r^2(t_i^2-1)\right)\tan^{-1}\left(\frac{t_r-1}{t_i}\right) \right. \\
    &\left. +\left((-1+t_r)^2+t_i^2\right)\left((1+t_r)^2+t_i^2\right)\tan^{-1}\left(\frac{1+t_r}{t_i}\right)\right),
    \end{aligned} & \text{ if } t_i\neq 0, \\
    \dfrac{2(1+3t_r^2)}{3(t_r^2-1)^3},&\text{ if } t_i=0,
    \end{cases}
    \label{eq:nu_rec_initial_value_3}
\end{equation}
\begin{equation}
    \nu_2^{5/2}(\trootlambda) = 
    \begin{cases}
    \dfrac{t_r}{2t_i^3}\left(-\dfrac{2t_i(-1+|\trootlambda|^2)}{t_r^4+2t_r^2(-1+t_i^2)+(1+t_i^2)^2} + \tan^{-1}\left(\dfrac{1-t_r}{t_i}\right)+ \tan^{-1}\left(\dfrac{1+t_r}{t_i}\right)\right),&\text{ if } t_i\neq 0, \\
    \dfrac{8t_r}{3(t_r^2-1)^3},&\text{ if } t_i=0.
    \end{cases}
    \label{eq:nu_rec_initial_value_4}
\end{equation}
\end{proof}

\section{Roots of planar squared-distance functions}\label{a:roots}This appendix derives the analytical formulas used in Section \ref{s:root_finding} for computing the complex root of the squared-distance function $R_\lambda^2$ in spheroidal geometries. The key observation is that the spheroidal problem can be reduced to finding the roots of squared-distance functions for circles and ellipses in the plane.

We first recall the known result for a circle, which was previously derived in \cite{SORGENTONE2023}. We then extend it to ellipses using a Joukowsky transform that maps ellipses to circles.

\begin{lemma}[Root of $R^2$ for circle in the plane \cite{SORGENTONE2023}]\label{lem:root_circle}Let a circle of radius $a>0$ in the $xy$-plane be parameterized by
\begin{equation}
    \ggamma(\alpha)=a\left(\cos(\alpha),\sin(\alpha),0\right),\quad 0\leq \alpha<2\pi.
\end{equation}
Given a point $\xx=(x,y,z)\in\mathbb{R}^3$ with $\rho=\sqrt{x^2+y^2}>0$, not on the circle, define
\begin{equation}
    R^2(\alpha,\xx) = \left(a\cos(\alpha)-x\right)^2 + \left(a\sin(\alpha)-y\right)^2 + z^2.
\end{equation}
Then $R^2(\alpha,\xx)=0$ for $\alpha=\alpha_0$, with
\begin{equation}
    \alpha_0 = \atantwo(y,x) \pm i\ln\left(\frac{1}{\rho}\left(\lambda+\sqrt{\lambda^2-\rho^2}\right)\right),
\end{equation}
where
\begin{equation}
    \lambda = \frac{1}{2a}\left(a^2+x^2+y^2+z^2\right),\quad \lambda>\rho.
\end{equation}
Here $\atantwo(\eta,\xi)$ is the argument of the complex number $\xi+i\eta$, $-\pi<\atantwo(\eta,\xi)\leq\pi$.
\end{lemma}

\begin{lemma}[Root of $R^2$ for ellipse in the plane]\label{lem:root_ellipse}Let an ellipse with semi-axes $a,b>0$ in the $xy$-plane be parameterized by 
\begin{equation}
    \ggamma(\alpha)=\left(a\cos(\alpha),b\sin(\alpha),0\right),\quad 0\leq \alpha<2\pi.
\end{equation}
Given a point $\xx=(x,y)\in\mathbb{R}^2$ in the plane of the ellipse and not on it, define
\begin{equation}
    R^2(\alpha,\xx) = \left(a\cos(\alpha)-x\right)^2 + \left(b\sin(\alpha)-y\right)^2.
\end{equation}
Then, $R^2(\alpha,\xx)=0$ for $\alpha=\alpha_0$, with
\begin{equation}
    \alpha_0 = \atantwo(\ytilde,\xtilde) \pm i\ln\left(\frac{1}{\rho}\left(\lambda+\sqrt{\lambda^2-\rho^2}\right)\right),
\end{equation}
where
\begin{equation}
\begin{split}
    \atilde = \frac{a+b}{2},\quad c^2 = \frac{a^2-b^2}{4}, \quad w = x+iy, \\
    u = \frac{1}{2}\left(w\pm\sqrt{w^2-4c^2}\right), \quad\xtilde=\Re(u), \quad\ytilde = \Im(u),
\end{split}
\end{equation}
and
\begin{equation}
    \lambda = \frac{1}{2a}\left(\atilde^2+\xtilde^2+\ytilde^2\right).
\end{equation}
Here $\atantwo(\eta,\xi)$ is defined as in Lemma \ref{lem:root_circle}.
\end{lemma}

\begin{proof}
Consider the Joukowsky transform
\begin{equation}
    w = w(u) = u + \frac{c^2}{u},
    \label{eq:joukowsky}
\end{equation}
where $u\in\mathbb{C}\setminus\{0\}$ for some constant $c$. The Joukowsky transform is a conformal mapping if $u\neq0$ and $u\neq\pm1$, which maps a circle from the $u$-plane into a transformed shape (in this case an ellipse), in the $w$-plane. Conversely, its inverse,
\begin{equation}
    u = \frac{1}{2}\left(w\pm\sqrt{w^2-4c^2}\right),
    \label{eq:inverse_joukowsky}
\end{equation}
maps in the other direction.

Consider the circle $\atilde\left(\cos(\alpha),\sin(\alpha),0\right)$ with radius $\atilde>0$ in the $xy$-plane. Using \eqref{eq:joukowsky} it is mapped into the ellipse $\left(a\cos(\alpha),b\sin(\alpha),0\right)$ in the same plane with semi-axes
\begin{equation}
    a = \atilde + \frac{c^2}{\atilde},\quad b = \atilde - \frac{c^2}{\atilde}.
    \label{eq:a_b_lemma}
\end{equation}
Solving these relations yields
\begin{equation}
    \atilde = \frac{a+b}{2},\quad c^2=\frac{a^2-b^2}{4}.
\end{equation}
Thus, using the inverse map \eqref{eq:inverse_joukowsky}, a point $wu=x+iy$ in the ellipse coordinates corresponds to the point $u=\xtilde+i\ytilde$ in the circle coordinates. The situation therefore reduces to the circle case considered in Lemma \ref{lem:root_circle}, applied to the circle radius $\atilde$ and the point $(\xtilde,\ytilde)$. Substituting the resulting expressions yields the stated formula.
\end{proof}

\section{Stabilization of recurrence relations}\label{a:rec_stability}As described in Remark \ref{rem:mu_rec_stability}, the forward recurrences for $\mu_k^p(\chi)$ in \eqref{eq:mu_rec} can become numerically unstable when $\chi$ is small. The desired sequence is the minimal solution and decays approximately like $\chi^k$. The recurrence also admits a complementary solution that grows with $k$. In finite precision, even a small roundoff component in the growing solution can eventually dominate the computed values and obscure the desired solution.

When this instability is detected, we use the stabilization produce of \cite{ARNOLDUS1984}, in which the three-term recurrence is recast as a homogeneous tridiagonal boundary-value problem of size $k_0-1$. The endpoint $k_0$ must be chosen large enough to include all Fourier modes required by the truncated expansion, and sufficiently far out that the minimal solution has decayed to a negligible size.

We use the following asymptotic model for the magnitude of the minimal solution:
\begin{equation}
\mu_k^p(\chi)\approx \Tilde{\mu}_k^p(\chi) \coloneqq C(p)\mu_0^p(\chi)\chi^k, \qquad 0<\chi<1.
\label{eq:mu_est}
\end{equation}
The constants $C(p)$ are chosen empirically for $p=1/2,3/2,5/2$. This model is not used as a rigorous bound, but as a practical rule for choosing the size of the tridiagonal system and for deciding when the forward recurrence is expected to be reliable. The corresponding decay models and roundoff-growth indicators $\Erecest(\chi)$ are listed in Table \ref{tab:decay_rate_rec_err}. In the implementation, the forward recurrence is used only if $\Erecest(\chi)$, evaluated at the largest required mode, is below the requested recurrence tolerance. Otherwise, the tridiagonal stabilization is used.

Figure \ref{fig:mu3_decay_rate_rec_err} compares the model \eqref{eq:mu_est} and the forward recurrence error predictor in Table \ref{tab:decay_rate_rec_err} with accurate values of $\mu_k^{3/2}(\chi)$ and with measured forward-recurrence errors. The agreement is sufficient for their intended roles: choosing the endpoint $k_0$ and deciding when to switch from the forward recurrence to the stabilized tridiagonal solve. The same behavior is observed for $p=1/2$ and $p=5/2$.

\begin{figure}[t]
\centering
\begin{subfigure}[t]{0.43\textwidth}
\includegraphics[width=\textwidth]{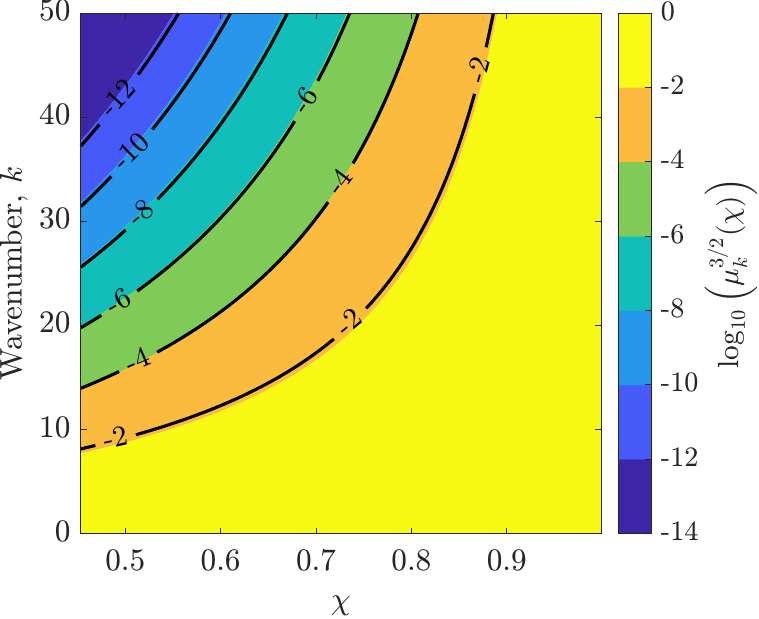}
\caption{}
\label{fig:mu3_decay_rate}
\end{subfigure}
\begin{subfigure}[t]{0.43\textwidth}
\includegraphics[width=\textwidth]{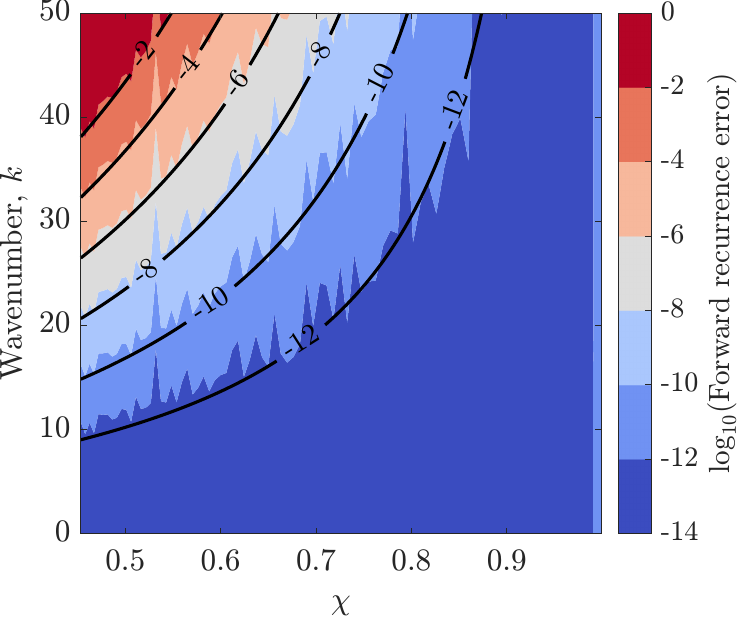}
\caption{}
\label{fig:mu3_rec_err}
\end{subfigure}
\caption{Panel (a) shows the magnitude of $\mu_k^{3/2}(\chi)$, with black contours showing the decay model from Table~\ref{tab:decay_rate_rec_err}. Panel (b) shows the absolute error produced by the forward recurrence for $p=3/2$, with black contours showing the corresponding roundoff-growth indicator $\E_{\textrm{rec}}^{3/2}$. In both panels, contours agree well with the measured quantities.}
\label{fig:mu3_decay_rate_rec_err}
\end{figure}

\begin{table}[t]
\centering
\caption{Decay models $\Tilde{\mu}_k^p(\chi)$ for $\mu_k^p(\chi)$ and corresponding roundoff-growth indicators for the forward recurrence \eqref{eq:mu_rec}, for $p=1/2,3/2,5/2$.}
\renewcommand{\arraystretch}{1.5}
\begin{tabular}{|c|c|c|c|}
\hline
Singularity order
& $p=1/2$
& $p=3/2$
& $p=5/2$
\\ \hline\hline
$\Tilde{\mu}_k^p(\chi)$
& $2^{-3}\mu_0^p(\chi)\chi^k$
& $2^{2}\mu_0^p(\chi)\chi^k$
& $2^{4}\mu_0^p(\chi)\chi^k$
\\ \hline
$\Erecest(\chi)$
& $\dfrac{10^{-16}}{2^{-2}\mu_0^p(\chi)\chi^k}$
& $\dfrac{10^{-14}}{2^{3}\mu_0^p(\chi)\chi^k}$
& $\dfrac{10^{-13}}{2^{6}\mu_0^p(\chi)\chi^k}$
\\ \hline
\end{tabular}
\label{tab:decay_rate_rec_err}
\end{table}

It remains to choose the endpoint $k_0$ of the tridiagonal system. Let $\epsilon_{\textrm{rec}}$ denote the recurrence tolerance, which in the experiments is taken to be the requested accuracy tolerance. From the decay model \eqref{eq:mu_est}, define
\begin{equation}
\kbar = \left\lceil\dfrac{\log\left(\dfrac{\epsilon_{\textrm{rec}}}{C(p)\mu_0^p(\chi)}\right)}{\log(\chi)}\right\rceil.
\end{equation}
Thus, according to the model, the minimal solution has then decayed to approximately $\epsilon_{\textrm{rec}}$ by mode $\kbar$. We then choose
\begin{equation}
k_0 = \max\left\{\left\lceil S\kbar\right\rceil,\,\kmax+1,\,2\right\},
\label{eq:k0}
\end{equation}
where $\kmax$ is the largest Fourier mode required by the truncated expansion \eqref{eq:Ip_approx_ssq}, and $S\geq1$ is a safety factor. In the experiments we use $S=1.5$.

\bibliographystyle{siamplain}
\bibliography{references}

\end{document}